\documentclass[11pt]{article}
\usepackage[pdftex,bookmarksopen=true,bookmarks=true,unicode,setpagesize]{hyperref}
\hypersetup{colorlinks=true,linkcolor=black,citecolor=black}

\usepackage{amssymb, amsmath, amsthm}

\allowdisplaybreaks
\usepackage{cite}

\newcommand\R{\mathbb{R}}

\newtheorem{theorem}{Theorem}[section]

\newtheorem{proposition}[theorem]{Proposition}

\theoremstyle{remark}

\theoremstyle{remark}
\newtheorem{example}[theorem]{Example}
\theoremstyle{remark}
\newtheorem{remark}[theorem]{Remark}

\begin{document}

\begin{center}{\Large \bf
Hafnian point processes and quasi-free states on the CCR algebra}
\end{center}

{\large Maryam Gharamah Ali Alshehri}\\ Department of Mathematics, Faculty of Science, University of Tabuk, Tabuk, KSA; \\
e-mail: \texttt{mgalshehri@ut.edu.sa}\vspace{2mm}

{\large Eugene Lytvynov\\ Department of Mathematics, Swansea University,  Swansea, UK;\\
e-mail: \texttt{e.lytvynov@swansea.ac.uk}\vspace{2mm}

{\small

\begin{center}
{\bf Abstract} 
\end{center}

\noindent  Let $X$  be a locally compact Polish space and $\sigma$  a nonatomic reference measure on $X$ (typically $X=\mathbb R^d$ and $\sigma$ is the Lebesgue measure). Let $X^2\ni(x,y)\mapsto\mathbb K(x,y)\in\mathbb C^{2\times 2}$ be a $2\times 2$-matrix-valued kernel that satisfies $\mathbb K^T(x,y)=\mathbb K(y,x)$. We say that a point process $\mu$ in $X$ is  hafnian with correlation kernel $\mathbb K(x,y)$ if,  for each $n\in\mathbb N$, the $n$th correlation function of $\mu$ (with respect to $\sigma^{\otimes n}$) exists and is given by $k^{(n)}(x_1,\dots,x_n)=\operatorname{haf}\big[\mathbb K(x_i,x_j)\big]_{i,j=1,\dots,n}\,$. Here $\operatorname{haf}(C)$ denotes the hafnian of a symmetric matrix $C$. Hafnian point processes include permanental and 2-permanental point processes as special cases. A Cox process $\Pi_R$ is a Poisson point process in $X$ with random intensity $R(x)$. Let $G(x)$ be a complex Gaussian field on $X$ satisfying $\int_{\Delta}\mathbb E(|G(x)|^2)\sigma(dx)<\infty$ for each compact $\Delta\subset X$. Then the Cox process $\Pi_R$ with $R(x)=|G(x)|^2$ is a hafnian point process. The main result of the paper is that each such process $\Pi_R$ is the joint spectral measure of a rigorously defined particle density of a representation of the canonical commutation relations (CCR), in a symmetric Fock space, for which the corresponding vacuum state on the CCR algebra is quasi-free. 

 } \vspace{2mm}

\vspace{2mm}


\section{Introduction}

\subsection{Hafnian point processes}

Let $X$ be a locally compact Polish space, let $\mathcal B(X)$ denote the Borel $\sigma$-algebra on $X$, and let $\mathcal B_0(X)$ denote the algebra of all pre-compact sets from $\mathcal B(X)$. Let $\sigma$ be a reference measure on $(X,\mathcal B(X))$ which is non-atomic (i.e., $\sigma(\{x\})=0$ for all $x\in X$) and Radon (i.e., $\sigma(\Delta)<\infty$ for all $\Delta\in\mathcal B_0(X)$). For applications, the most important example is $X=\R^d$ $\sigma(dx)=dx$ is the Lebesgue measure.

A {\it (simple) configuration} $\gamma$ in $X$ is a Radon measure on $X$ of the form $\gamma=\sum_i\delta_{x_i}$, where $\delta_{x_i}$ denotes the Dirac measure with mass at $x_i$ and $x_i\ne x_j$ if $i\ne j$. Note that, since $\gamma$ is a Radon measure, it has a finite number of atoms in each compact set in $X$. Let $\Gamma(X)$ denote the set of all configurations $\gamma$ in $X$. Let $\mathcal C(\Gamma(X))$ denote the minimal $\sigma$-algebra on $\Gamma(X)$ such that, for each $\Delta\in\mathcal B_0(X)$, the  mapping $\Gamma(X)\ni\gamma\mapsto\gamma(\Delta)$ is measurable. A {\it (simple) point process in $X$} is a probability measure on $(\Gamma(X),\mathcal C(\Gamma(X)))$.

Denote $X^{(n)}:=\{(x_1,\dots,x_n)\in X^n\mid x_i\ne x_j\text{ if }i\ne j\}$. A measure on $X^{(n)}$ is called symmetric if it remains invariant under the natural action of the symmetric group $\mathfrak S_n$ on $X^{(n)}$. 
 For each $\gamma=\sum_i\delta_{x_i}\in\Gamma(X)$, the {\it spatial falling factorial} $(\gamma)_n$ is the symmetric measure on $X^{(n)}$ of the form
\begin{equation}\label{ctrsw5ywus}
(\gamma)_n:=\sum_{i_1}\sum_{i_2\ne i_1}\dotsm\sum_{i_n\ne i_1,\dots, i_n\ne i_{n-1}}\delta_{(x_{i_1}, x_{i_2}, \dots,x_{i_n})}.\end{equation}
Let $\mu$ be a point process in $X$. The {\it $n$-th correlation measure of $\mu$} is the symmetric measure $\theta^{(n)}$ on $X^{(n)}$ defined by
\begin{equation}\label{rqay45qy4q}
\theta^{(n)}(dx_1\dotsm dx_n):=\frac1{n!}\int_{\Gamma(X)}(\gamma)_n(dx_1\dotsm dx_n)\,\mu(d\gamma).
\end{equation}
If each measure $\theta^{(n)}$ is absolutely continuous with respect to $\sigma^{\otimes n}$, then the symmetric functions $k^{(n)}:X^{(n)}\to[0,\infty)$ satisfying 
\begin{equation}\label{5w738}
d\theta^{(n)}=\frac1{n!}\,k^{(n)}d\sigma^{\otimes n}\end{equation} are called the {\it correlation functions of the point process $\mu$}. Under a very weak assumption, the correlations functions (or correlation measures) uniquely identify a point process, see  \cite{Lenard}.

Let $C=[c_{ij}]_{i,j=1,\dots,2n}$ be a symmetric $2n\times2n$-matrix. The {\it hafnian of $C$} is defined by
$$\operatorname{haf}(C):=\frac1{n!\,2^n}\sum_{\pi\in\mathfrak S_{2n}}\prod_{i=1}^n c_{\pi(2i-1)\pi(i)}, $$
see e.g.\ \cite[Section~4.1]{Barvinok}. 
(Note the the value of the hafnian of $C$ does not depend on the diagonal elements  of the matrix $C$.)
The hafnian can also be written as
\begin{equation}\label{cxtseewu645}
\operatorname{haf}(C)=\sum c_{i_1j_1}\dotsm c_{i_nj_n},\end{equation}
where the summation is over all (unordered) partitions $\{i_1,j_1\},\dots,\{i_n,j_n\}$ of $\{1,\dots,2n\}$. 

Hafnians were introduced by  physicist Edoardo Caianiello in the 1950's, while visiting Niels Bohr's group in Copenhagen (whose latin name is Hafnia), as a Boson analogue of the formula expressing the
correlations of a quasi-free Fermi state.\footnote{We are grateful to the referee for sharing with us this historical fact.}

By analogy with the definition of a pfaffian point process (see e.g.\ \cite[Section~10]{Borodin} and the references therein), we  now define a hafnian point process. Let $X^2\ni(x,y)\mapsto\mathbb K(x,y)\in\mathbb C^{2\times 2}$ be a $2\times 2$-matrix-valued kernel that satisfies $\mathbb K^T(x,y)=\mathbb K(y,x)$.
We will say that a point process $\mu$ is {\it hafnian with correlation kernel $\mathbb K(x,y)$} if,  for each $n\in\mathbb N$, the $n$th correlation function of $\mu$ exists and is given by 
\begin{equation}\label{vytds6}
k^{(n)}(x_1,\dots,x_n)=\operatorname{haf}\big[\mathbb K(x_i,x_j)\big]_{i,j=1,\dots,n}\,.\end{equation}
Note that the matrix
$$ \big[\mathbb K(x_i,x_j)\big]_{i,j=1,\dots,n}=\left[\begin{matrix}
\mathbb K(x_1,x_1)&\mathbb K(x_1,x_2)&\dotsm&\mathbb K(x_1,x_n)\\
\mathbb K(x_2,x_1)&\mathbb K(x_2,x_2)&\dotsm&\mathbb K(x_2,x_n)\\
\vdots&\vdots&\vdots&\vdots\\
\mathbb K(x_n,x_1)&\mathbb K(x_n,x_2)&\dotsm&\mathbb K(x_n,x_n)
\end{matrix}\right]$$
is built upon $2\times2$-blocks $\mathbb K(x_i,x_j)$, hence it has dimension $2n\times 2n$. Furthermore, the condition $\mathbb K^T(x,y)=\mathbb K(y,x)$ ensures that the matrix $\big[\mathbb K(x_i,x_j)\big]_{i,j=1,\dots,n}$ is symmetric, and so its hafnian is a well-defined number.

Since 
$$X^2=\{(x,x)\mid x\in X\}\sqcup X^{(2)},$$
for the definition of a hafnian point process, it is  sufficient to assume that $\mathbb K(x,x)$ is defined for $\sigma$-a.a.\ $x\in X$, and the restriction of $\mathbb K(x,y)$ to $X^{(2)}$ is defined for $\sigma^{\otimes 2}$-a.a.\ $(x,y)\in X^{(2)}$. 

Note that, for the hafnian point process $\mu$, the correlation kernel $\mathbb K(x,y)$ is not uniquely determined by $\mu$. Indeed, since the hafnian of a matrix does not depend on its diagonal elements, formula \eqref{vytds6} implies that the correlation functions $k^{(n)}(x_1,\dots,x_n)$ do not depend on the diagonal elements of the $2\times 2$-matrices $\mathbb K(x,x)$ for $x\in X$. Hence, these elements can be chosen arbitrarily.

Let $\alpha\in\mathbb R$ and let $B=[b_{ij}]_{i,j=1,\dots,n}$ be an $n\times n$ matrix. The {\it $\alpha$-determinant of $B$} is defined by
\begin{equation}\label{tera5yw3}
\operatorname{det}_\alpha (B):=\sum_{\pi\in \mathfrak S_n}\prod_{i=1}^n\alpha^{n-\nu(\pi)}\,b_{i\,\pi(i)},\end{equation}
see  \cite{VJ,ST}. In formula \eqref{tera5yw3},   for $\pi\in \mathfrak S_n$, $\nu(\pi)$ denotes the number of cycles in the permutation $\pi$. In particular, for $\alpha=1$, $\operatorname{det}_1(B)$ is the usual permanent of $B$.

 A point process $\mu$ is called {\it $\alpha$-permanental (or  $\alpha$-determinantal) with correlation kernel $K:X^2\to\mathbb C$} if, for each $n\in\mathbb N$, the $n$th correlation function of $\mu$ exists and is given by 
$$k^{(n)}(x_1,\dots,x_n)=\operatorname{det}_\alpha [K(x_i,x_j)]_{i,j=1,\dots,n},$$
 \cite{ST}, see also \cite{E2}. For $\alpha=1$, one calls $\mu$ a {\it permanental point process}.

As easily follows from \cite[Section~4.1]{Barvinok} a permanental point process with correlation kernel $K(x,y)$ is hafnian  with correlation kernel
$$\mathbb K(x,y)=\left[\begin{matrix}0&K(x,y)\\K(y,x)&0\end{matrix}\right].$$
Furthermore, similarly to \cite[Proposition~1.1]{Frenkel}, we see that a $2$-permanental point process with a symmetric correlation kernel $K(x,y)=K(y,x)$ is hafnian with the correlation kernel
$$\mathbb K(x,y)=\left[\begin{matrix}K(x,y)&K(x,y)\\K(x,y)&K(x,y)\end{matrix}\right].$$
For studies of permanental, and more generally $\alpha$-permanental point processes, we refer to \cite{BM,KMR,Macchi1,Macchi2,ST}.

Recall that a {\it Cox  process $\Pi_R$} is a Poisson point process with a random intensity $R(x)$. Here $R(x)$ is a random field defined for $\sigma$-a.a\ $x\in X$ and taking a.s.\ non-negative values. The  correlation functions of the Cox process $\Pi_R$ are  given by
\begin{equation}\label{xreew5u}
k^{(n)}(x_1,\dots,x_n)=\mathbb E\big(R(x_1)\dotsm R(x_n)\big).\end{equation}

Let $G(x)$ be a mean-zero, complex Gaussian field defined for $\sigma$-a.a.\ $x\in X$. Assume additionally that $\int_\Delta\mathbb E(|G(x)|^2)\sigma(dx)<\infty$  for each $\Delta\in\mathcal B_0(X)$. Let $R(x):=|G(x)|^2=G(x)\overline{G(x)}$. 
Comparing the classical moment formula for Gaussian random variables  with formula \eqref{cxtseewu645}, we immediately see that
\begin{equation}\label{dr6e6u43wq}
\mathbb E\big(R(x_1)\dotsm R(x_n)\big)=\operatorname{haf}\big[\mathbb K(x_i,x_j)\big]_{i,j=1,\dots,n}\, , \end{equation}
where 
\begin{equation}\label{d6e6ie4}
\mathbb K(x,y)=\left[\begin{matrix}
\mathbb E(G(X)G(y))&\mathbb E(G(x)\overline{G(y)})\\\mathbb E(\overline{G(x)}G(y))&\mathbb E(\overline{G(x)G(y)})\end{matrix}\right]=\left[\begin{matrix}
\mathcal K_2(x,y)&\mathcal K_1(x,y)\\\overline{\mathcal K_1(x,y)}&\overline{\mathcal K_2(x,y)}\end{matrix}\right].\end{equation}
Here $\mathcal K_1(x,y):=\mathbb E(G(x)\overline{G(y)})$ is the {\it covariance} of the Gaussian field and $\mathcal K_2(x,y):=\mathbb E(G(x)G(y))$ is the {\it pseudo-covariance} of the Gaussian field. By \eqref{xreew5u}--\eqref{d6e6ie4}, the corresponding Cox process  $\Pi_R$ is hafnian with the correlation kernel \eqref{d6e6ie4}.

In the case where the Gaussian field $G(x)$ is real-valued, the moments
of $R(x)$ are given by the $2$-determinants built upon the kernel $K(x,y):=\mathcal K_1(x,y)=\mathcal K_2(x,y)$, hence $R(x)$ is a 2-permanental process.  For studies of $\alpha$-permanental processes, we refer e.g.\ to \cite{E1,E2,E3,E4,KMR,MR1,MR2,MR3} and the references therein.
Obviously, in this case, $\Pi_R$ is  a $2$-permanental point process with the correlation kernel $K(x,y)$, compare with \cite[Subsection~6.4]{ST}.

A Gaussian random field is called {\it proper} if $\mathcal K_2(x,y)=0$ for all $x$ and $y$. Since the moments of the random field $R(x)$ are given by permanents built upon the kernel $K(x,y):=\mathcal K_1(x,y)$, $R(x)$ is  a  permanental process, compare with \cite{BM,Macchi1,Macchi2}. We note, however,  that the available studies of $\alpha$-permanental processes usually discuss only the case where the kernel  is real-valued. In the case of $R(x)$, the correlation kernel is, of course, complex-valued.

\subsection{Aim  of the paper}

Quasi-free states play a central role in studies of operator algebras related to quantum statistical mechanics, see e.g.\  \cite{A1,A2,A3,A4,BR,DG,MV}.  

Let $\mathcal H=L^2(X,\sigma)$  be the $L^2$-space of $\sigma$-square-integrable functions $h:X\to\mathbb C$. Let  $\mathfrak F$ be  a separable complex  Hilbert spaces. Let $A^+(h)$,   $A^-(h)$ ($h\in\mathcal H$) be linear operators in $\mathfrak F$ that satisfy the following assumptions: 
\begin{itemize}

\item[(i)] $ A^+(h)$ and $A^-(h)$ depend linearly on $h\in\mathcal H$;

\item[(ii)] for each  $h\in\mathcal H$, $A^-(\overline h)$ is  (the restriction of) the adjoint operator  of $A^+(h)$, where $\bar h$ is the complex conjugate of $h$; 

\item[(iii)] the operators $A^+(h)$,   $A^-(h)$ satisfy the canonical commutation relations (CCR).

\end{itemize}
See Section~\ref{xrwaq4q2} for details.

 Let $\mathbb A$ be the unital $*$-algebra generated by the operators $A^+(h)$,   $A^-(h)$. If we additionally assume that $\mathfrak F$ is a certain symmetric Fock space, then we can define the vacuum state $\tau$ on   $\mathbb A$. If $\tau$ appears to be a quasi-free state, one says that the operators $A^+(h)$ and $A^-(h)$ form a {\it quasi-free representation of the CCR}.

We  define operator-valued distributions $A^+(x)$ and $A^-(x)$
($x\in X$) through the equalities 
 \begin{equation}\label{ray45}
A^+(h)=\int_X h(x)A^+(x)\sigma(dx),\quad A^-(h)=\int_X h(x)A^-(x)\sigma(dx),\end{equation}
holding for all $h\in\mathcal H$.

Then the {\it particle density $\rho(x)$} is formally defined as 
$$\rho(x):=A^+(x)A^-(x),\quad x\in X.$$  
We called this definition {\it formal} since it requires to take product of two operator-valued distributions, and {\it a priori\/} it is  not clear if this product indeed makes sense. Nevertheless, in all the examples below, we will be able to rigorously define $\rho(x)$ as an operator-valued distribution.

The CCR  imply   the commutation $[\rho(x),\rho(y)]=0$ ($x,y\in X$), where  $[\cdot,\cdot]$ denotes the commutator. For each $\Delta\in\mathcal B_0(X)$, we denote
\begin{equation}\label{raq5wu}
\rho(\Delta):=\int_\Delta\rho(x)\sigma(dx)=\int_\Delta A^+(x)A^-(x)\sigma(dx),\end{equation}
which is  a family of Hermitian commuting operators in the Fock space $\mathfrak F$. In view of the spectral theorem, one can expect that the operators $(\rho_\Delta)_{\Delta\in\mathcal B_0(X)}$ can be realized as operators of multiplication in  $L^2(\Gamma_X,\mu)$, where $\mu$ is the joint spectral measure of this family of operators at the vacuum.  

Let $G(x)$ be a complex-valued  Gaussian field and $R(x)=|G(x)|^2$. The main aim of the paper is show that the Cox process $\Pi_R$ is  the joint spectral measure of
a (rigorously defined) particle  density $(\rho_\Delta)_{\Delta\in\mathcal B_0(X)}$ for a certain quasi-free representation of the CCR. As a by-product, we obtain a unitary isomorphism between a subspace of a Fock space and $L^2(\Gamma_X,\Pi_R)$. 

In the special case where $\Pi_R$ is a permanental point process (with a real-valued correlation kernel), such a statement was proved in \cite{LM} (see also \cite{Ly}).
In that case, the corresponding quasi-free state has an additional property of being gauge-invariant, so one could use the gauge-invariant  quasi-free representation of the CCR by Araki and Woods \cite{ArWoods}. 

We stress that, even in the case of a gauge-invariant quasi-free state, the representation of the CCR  that we use in this paper has a different form as compared to the  one by Araki and Woods \cite{ArWoods}. Nevertheless, since both gauge-invariant quasi-free representations have the same $n$-point functions, one can show that these representations are unitarily equivalent.

We note that, in \cite{Ly,LM}, it was also shown that each determinantal point process ($\alpha=-1$) arises as the joint spectral measure of the particle density of a quasi-free representation of the Canonical Anticommutation Relations (CAR). In that case, the state is also gauge-invariant, so one can use the Araki--Wyss representation of the CAR from \cite{ArWyss}.

  It is worth to compare our result with the main result of Koshida \cite{Koshida}. In the latter paper, it is proven that, when the  underlying space $X$ is {\it discrete},   every pfaffian point process on $X$ arises as the particle density of a quasi-free representation of the CAR. As noted in  \cite{Koshida}, a similar statement in the case of a continuous space $X$ is still an open problem.

\subsection{Organization of the paper}

The starting point of our considerations is the observation that the Poisson point process with (deterministic) intensity $|\lambda(x)|^2$ arises from the trivial (quasi-free) representation of the CCR with 
\begin{equation}\label{teraw5uw}
A^+(x)=a^+(x)+\overline{\lambda(x)},\quad A^-(x)=a^-(x)+\lambda(x),
\end{equation}
 where $a^+(x)$, $a^-(x)$ are the creation and annihilation operators at point $x$, acting in the symmetric Fock space $\mathcal F(\mathcal H)$ over $\mathcal H$, compare with \cite{GGPS}. We then proceed as follows:

\begin{itemize}
\item We realize a Gaussian field $G(x)$ as a family of operators $\Phi(x)$
acting in a Fock space $\mathcal F(\mathcal G)$ over a Hilbert space $\mathcal G$ (typically $\mathcal G=\mathcal H$  or $\mathcal G=\mathcal H\oplus\mathcal H$).

\item We consider a quasi-free representation of the CCR with 
\begin{equation}\label{fyr6sw6u}
A^+(x)=a^+(x)+\Phi^*(x),\quad  A^-(x)=a^-(x)+\Phi(x)
\end{equation}
 acting in the Fock space $\mathcal F(\mathcal H\otimes\mathcal G)=\mathcal F(\mathcal H)\otimes\mathcal F(\mathcal G)$.

\item We prove that the corresponding particle density $(\rho_\Delta)_{\Delta\in\mathcal B_0(X)}$ is well-defined and has the joint spectral measure $\Pi_R$.

\end{itemize}

The paper is organized as follows. In Section \ref{ew56u3wu}, we discuss complex-valued Gaussian fields on $X$ realized in a symmetric Fock space $\mathcal F(\mathcal G)$ over a separable Hilbert space $\mathcal G$. We start with a $\mathcal G^2$-valued function $(L_1(x),L_2(x))$ that is defined for $\sigma$-a.a.\ $x\in X$ and satisfies the assumptions \eqref{ydxdrdxdrg}, \eqref{xrsa5yw4} below. We then define operators $\Phi(x)$ in the Fock space $\mathcal F(\mathcal G)$   by formula \eqref{waaq4yq5y}. Theorem~\ref{reaq5y43wu} states that the operators $\Phi(x)$ form a  Fock-space realization of a Gaussian field $G(x)$ that is defined for $\sigma$-a.a.\ $x\in X$. (Note, however, that the set of those $x\in X$ for which $G(x)$ is defined can be smaller than the set of those $x\in X$ for which the function $(L_1(x),L_2(x))$ was defined.) The covariance and  pseudo-covariance of the Gaussian field $G(x)$  are given by formulas \eqref{w909i8u7y689} and \eqref{ftsqw43qd}, respectively. 

As a consequence of our considerations, in Example~\ref{vcrtw5y3}, we derive a  Fock-space realization of a proper Gaussian field. The operators $\Phi(x)$ in this case resemble the classical Fock-space realization of a real-valued Gaussian field. The main difference is that, in the case of a real-valued Gaussian field, the creation and annihilation operators use  same real vectors, whereas in the case of a proper Gaussian field, the creation and annihilation operators use orthogonal copies of  same complex vectors.

In Section~\ref{xrwaq4q2}, we briefly recall the definition of a quasi-free state on the CCR algebra and a quasi-free representation of the CCR. 

Next, in Section~\ref{tew532w}, we recall in Theorem~\ref{tes56uwe4u6} a result from \cite{LM}  which gives sufficient conditions for a family of commuting Hermitian operators, $(\rho(\Delta))_{\Delta\in\mathcal B_0(X)}$, in a separable complex Hilbert space, to be essentially self-adjoint and have a point process $\mu$ in $X$ as their joint spectral measure. The key condition of Theorem~\ref{tes56uwe4u6} is that the family of operators, $(\rho(\Delta))_{\Delta\in\mathcal B_0(X)}$,  should possess certain correlation measures $\theta^{(n)}$, whose definition is given in Section~\ref{tew532w}. These measures $\theta^{(n)}$ are  then also the correlation measures of the point process $\mu$. We also present formal considerations about the form of the  correlation measures $\theta^{(n)}$ when $\rho(\Delta)$ is a particle density given by \eqref{raq5wu}.

In Section~\ref{xeeraq54q}, we apply Theorem~\ref{tes56uwe4u6} to show that a Poisson point process is  the joint spectral measure of the operators  $(\rho(\Delta))_{\Delta\in\mathcal B_0(X)}$, where $\rho(\Delta)$ is the particle density of the trivial quasi-free representation of the CCR in which the creation and annihilation operators are given by \eqref{teraw5uw}. 

The main results of the paper are in Section~\ref{yd6w6wdd}. Using the $\mathcal G^2$-valued function $(L_1(x),L_2(x))$ from Section~\ref{ew56u3wu}, we construct a quasi-free representation of the CCR in the symmetric Fock space $\mathcal F(\mathcal H\oplus\mathcal G)$. We prove that the corresponding particle density is well defined as a family of commuting Hermitian operators,  $(\rho(\Delta))_{\Delta\in\mathcal B_0(X)}$ (Corollary~\ref{3rtlgpr}). Theorem~\ref{due6uew4} states that these operators satisfy the assumptions of Theorem~\ref{tes56uwe4u6} and their joint spectral measure $\mu$ is the Cox process $\Pi_R$, where  $R(x)=|G(x)|^2$ and $G(x)$ is the Gaussian field as in in Theorem~\ref{reaq5y43wu}. In particular, $\mu$ is a hafnian point process.

\section{Fock-space realization of complex Gaussian fields}\label{ew56u3wu}

Let $\mathcal G$ be a separable Hilbert space with an antilinear involution $\mathcal J$ satisfying $(\mathcal Jf,\mathcal Jg)_{\mathcal G}=(g,f)_{\mathcal G}$ for all $f,g\in\mathcal G$. Let $\mathcal G^{\odot n}$ denote the $n$th symmetric tensor power of $\mathcal G$. For $n\in\mathbb N$, let $\mathcal F_n(\mathcal G):=\mathcal G^{\odot n}n!$, i.e., $\mathcal F_n(\mathcal G)$ coincides with $\mathcal G^{\odot n}$ as a set and the inner product in $\mathcal F_n(\mathcal G)$ is equal to $n!$ times the inner product in $\mathcal G^{\odot n}$. Let also $\mathcal F_0(\mathcal G):=\mathbb C$. Then $\mathcal F(\mathcal G)=\bigoplus_{n=0}^\infty\mathcal F_n(\mathcal G)$ is called the {\it symmetric Fock space over $\mathcal G$}. The vector $\Omega=(1,0,0,\dots)\in\mathcal F(\mathcal G)$ is called the {\it vacuum}. 

Let $\mathcal F_{\mathrm{fin}}(\mathcal G)$ denote the (dense) subspace of  $\mathcal F(\mathcal G)$ consisting of all finite vectors $f=(f^{(0)},f^{(1)},\dots,f^{(n)},0,0,\dots)$ ($n\in\mathbb N$). We equip $\mathcal F_{\mathrm{fin}}(\mathcal G)$ with the topology of the topological direct sum of the $\mathcal F_n(\mathcal G)$ spaces.

For topological vector spaces $V$ and $W$, we denote by $\mathcal L(V,W)$ the space of all linear continuous operators $A:V\to W$. We also denote $\mathcal L(V):=\mathcal L(V,V)$.

Let $g\in\mathcal G$. We define a {\it creation operator} $a^+(g)\in\mathcal L(\mathcal F_{\mathrm{fin}}(\mathcal G))$ by $a^+(g)\Omega:=g$, and for $f^{(n)}\in\mathcal F_n(\mathcal G)$ ($n\in\mathbb N$), 
$a^+(g)f^{(n)}:=g\odot f^{(n)}\in\mathcal F_{n+1}(\mathcal G)$ . Next, we define an {\it  annihilation operator} $a^-(g)\in\mathcal L(\mathcal F_{\mathrm{fin}}(\mathcal G))$ that satisfies $a^-(g)\Omega:=0$ and  for any $f_1,\dots,f_n\in\mathcal G$,
$$a^-(g)f_1\odot\dotsm\odot f_n=\sum_{i=1}^n(f_i,\mathcal Jg)_{\mathcal G}\,f_1\odot\dots\odot f_{i-1}\odot f_{i+1}\odot\dots\odot f_n.$$
We have $a^+(g)^*\restriction_{\mathcal F_{\mathrm{fin}}(\mathcal G)}=a^-(\mathcal Jg)$ and the operators $a^+(g)$, $a^-(g)$ satisfy the CCR:
\begin{equation}\label{tsw654w}
[a^+(f),a^+(g)]=[a^-(f),a^-(g)]=0,\quad [a^-(f),a^+(g)]=(g,\mathcal Jf)_{\mathcal G}\end{equation}
for all $f,g\in\mathcal G$.

Let $D\in\mathcal B(X)$ be such that $\sigma(X\setminus D)=0$. Let $D\ni x\mapsto (L_1(x),L_2(x))\in\mathcal G^2$ be a measurable mapping. We assume that
\begin{align}
&(L_1(x),\mathcal JL_2(y))_{\mathcal G}=(L_1(y),\mathcal JL_2(x))_{\mathcal G},\label{ydxdrdxdrg}\\
&(L_1(x),L_1(y))_{\mathcal G}=(L_2(x),L_2(y))_{\mathcal G}\quad \text{for all }x,y\in D.\label{xrsa5yw4}
\end{align}
 Define
\begin{equation}\label{waaq4yq5y}
\Phi(x):=a^+(L_1(x))+a^-(L_2(x)).\end{equation}
Let $\Psi(x):=\Phi(x)^*\restriction_{\mathcal F_{\mathrm{fin}}(\mathcal G)}$.
Then
\begin{equation}\label{ytd7r57}
\Psi(x)=a^+(\mathcal JL_2(x))+a^-(\mathcal JL_1(x)).\end{equation}
It follows from \eqref{tsw654w} that conditions \eqref{ydxdrdxdrg}, \eqref{xrsa5yw4} are necessary and sufficient in order that  $[\Phi(x),\Phi(y)]=[\Psi(x),\Psi(y)]=[\Phi(x),\Psi(y)]=0$ for all $x,y\in D$.

Below, for each $\Lambda\subset D$, we denote by $\mathbb F_\Lambda$ the subspace of the Fock space $\mathcal F(\mathcal G)$ that is the closed linear span of the set
\begin{multline*}
\big\{\Psi(x_1)^{k_1}\dotsm\Psi(x_m)^{k_m}\Phi(y_1)^{l_1}\dotsm \Phi(y_n)^{l_n}\Omega\mid x_1,\dots,x_m,y_1,\dots,y_n\in \Lambda,\\ k_1,\dots,k_m,l_1,\dots,l_n\in\mathbb N_0,\ m,n\in\mathbb N\big\}.\end{multline*}
Here $\mathbb N_0:=\{0,1,2,3,\dots\}$.

\begin{theorem}\label{reaq5y43wu}
There exists a measurable subset $\Lambda\subset D$ with $\sigma(X\setminus\Lambda)=0$ and a mean-zero complex-valued Gaussian field $\{G(x)\}_{x\in \Lambda}$  on a probability space $(\Xi,\mathfrak A,P)$ such that:

\begin{itemize}
\item[(i)] The  Gaussian field $\{G(x)\}_{x\in \Lambda}$ has the covariance
\begin{equation}\label{w909i8u7y689}
\mathcal K_1(x,y)=(L_1(x),L_1(y))_{\mathcal G},\quad x,y\in \Lambda,\end{equation}
and the pseudo-covariance 
\begin{equation}\label{ftsqw43qd}
\mathcal K_2(x,y)=(L_1(x),\mathcal JL_2(y))_{\mathcal G},\quad x,y\in \Lambda.\end{equation}

\item[(ii)] There exists a unique unitary operator $\mathcal I:\mathbb F_\Lambda\to L^2(\Xi,P)$ that satisfies
\begin{align}
&\mathcal I\Psi(x_1)^{k_1}\dotsm\Psi(x_m)^{k_m}\Phi(y_1)^{l_1}\dotsm \Phi(y_n)^{l_n}\Omega\notag\\
&\quad= \overline{G(x_1)}^{\,k_1}\dotsm\overline{G(x_m)}^{\,k_m}G(y_1)^{l_1}\dotsm G(y_n)^{l_n}\label{cy6e64u43}\end{align}
for all $x_1,\dots,x_m,y_1,\dots,y_n\in \Lambda$, $k_1,\dots,k_m,l_1,\dots,l_n\in\mathbb N_0$, $m,n\in\mathbb N$.
\end{itemize}

\end{theorem}

\begin{proof} We define $L_1(x)=L_2(x)=0$ for all $x\in X\setminus D$. Then 
\begin{equation}\label{tes5w53w}
X\ni x\mapsto(L_1(x),L_2(x))\in\mathcal G^2
\end{equation}
 is measurable and satisfies \eqref{ydxdrdxdrg}, \eqref{xrsa5yw4} for all $x,y\in X$. By Lusin's theorem (see e.g.\ \cite[26.7~Theorem]{Bauer}), there exists a sequence of mutually disjoint compact sets $(\Lambda_n)_{n=1}^\infty$ such that $\sigma\big(X\setminus\bigcup_{n=1}^\infty\Lambda_n\big)=0$, and the restriction of the mapping \eqref{tes5w53w} to each $\Lambda_n$ is continuous. Denote $\Lambda:=\bigcup_{n=1}^\infty\Lambda_n$ and choose a countable subset $\Lambda'\subset\Lambda$ such that, for each $n\in\mathbb N$, the set $\Lambda'\cap\Lambda_n$ is dense in $\Lambda_n$. As easily seen by approximation, $\mathbb F_{\Lambda'}=\mathbb F_\Lambda$.

Let us consider the real and imaginary parts of the operators $\Phi(x)$: 
\begin{align}
\Re(\Phi(x))&:=\frac12(\Phi(x)+\Psi(x))=\frac12\left(a^+\big(L_1(x)+\mathcal JL_2(x)\big)+a^-\big(\mathcal JL_1(x)+L_2(x)\big)\right),\notag\\
\Im(\Phi(x))&:=\frac1{2i}(\Phi(x)-\Psi(x))=\frac1{2i}\left(a^+\big(L_1(x)-\mathcal JL_2(x)\big)-a^-\big(\mathcal JL_1(x)-L_2(x)\big)\right).\label{se5w5u33w}
\end{align}
These operators are Hermitian and commuting.

It is a standard fact that, for each $g\in\mathcal G$, 
\begin{equation}\label{cr6sw6u4e}
\|a^+(g)\|_{\mathcal L(\mathcal F_k(\mathcal G),\mathcal F_{k+1}(\mathcal G))}=\|a^-(g)\|_{\mathcal L(\mathcal F_{k+1}(\mathcal G),\mathcal F_{k}(\mathcal G))}=\sqrt{k+1}\,\|g\|_{\mathcal G}. \end{equation}
From here it easily follows that each $f\in\mathcal F_{\mathrm{fin}}(\mathcal G)$ is an analytic vector for  each $\Re(\Phi(x))$ and $\Im(\Phi(x))$ ($x\in X$), and the projection-valued measures of the closures of all these operators commute, see \cite[Chapter~5, Theorem~1.15]{BK}.

We now apply the projection spectral theorem \cite[Chapter~3, Theorems~2.6 and 3.9 and Section~3.1]{BK} to the closures of the operators $\Re(\Phi(x))$ and $\Im(\Phi(x))$ with $x\in\Lambda'$. This implies the existence of a probability space $(\Xi,\mathfrak A,P)$, real-valued random variables $G_1(x)$ and $G_2(x)$ ($x\in\Lambda'$) and a unique unitary operator $\mathcal I:\mathbb F_\Lambda\to L^2(\Xi,P)$ that satisfies
\begin{align}
&\mathcal I\,\Re(\Phi(x_1))^{k_1}\dotsm\Re(\Phi(x_m))^{k_m}\Im(\Phi(y_1))^{l_1}\dotsm \Im(\Phi(y_n))^{l_n}\Omega\notag\\
&\quad= G_1(x_1)^{\,k_1}\dotsm G_1(x_m)^{\,k_m}G_2(y_1)^{l_1}\dotsm G_2(y_n)^{l_n}
\label{xrq5y3q2u}
\end{align}
for all $x_1,\dots,x_m,y_1,\dots,y_n\in \Lambda'$, $k_1,\dots,k_m,l_1,\dots,l_n\in\mathbb N_0$, $m,n\in\mathbb N$. 

\begin{remark} In fact, $\Xi=\{\omega:\Lambda'\to\R^2\}$, $\mathfrak A$ is the cylinder $\sigma$-algebra on $\Xi$ (equivalently the countable product of the Borel $\sigma$-algebras $\mathcal B(\R)$), and   $P(\cdot)=(E(\cdot)\Omega,\Omega)_{\mathcal F(\mathcal G)}$. Here,  $E$ is the projection-valued measure on $(\Xi,\mathfrak A)$ that is constructed as the countable product of the projection-valued measures of the closures of the operators $\Re(\Phi(x))$ and  $\Im(\Phi(x))$ with $x\in\Lambda'$. Furthermore, for each $x\in\Lambda'$, $(G_i(x))(\omega)=\omega_i(x)$ for $\omega=(\omega_1,\omega_2)\in\Xi$. 
\end{remark}

Next, let $n\in\mathbb N$ and $x\in\Lambda_n\setminus\Lambda'$. Then we can find a sequence $(x_k)_{k=1}^\infty$ in $\Lambda'\cap \Lambda_n$ such that $x_k\to x$,  hence, by continuity, $(L_1(x_k),L_2(x_k))\to (L_1(x),L_2(x))$ in $\mathcal G^2$ . It follows from \eqref{xrq5y3q2u} that $(G_i(x_k))_{k=1}^\infty$ is a Cauchy sequence in $L^2(\Xi,P)$ ($i=1,2$), so we define $G_i(x):=\lim_{k\to\infty}G_i(x_k)$. Then we easily see by approximation that \eqref{xrq5y3q2u} remains true for all $x_1,\dots,x_m,y_1,\dots,y_n\in \Lambda$.

Let $Z$ be an arbitrary finite linear combination (with real coefficients) of random variables  from $\{G_1(x),\,G_2(x)\mid x\in\Lambda\}$. Then it follows from \eqref{se5w5u33w} that the moments of $Z$ are given by 
$$\mathbb E(Z^k)=\big((a^+(g)+a^-(\mathcal Jg))^k\,\Omega,\Omega\big)_{\mathcal F(\mathcal G)}$$
for some $g\in\mathcal G$. But this implies that the random variable $Z$ has a Gaussian distribution, see e.g.\ \cite[Chapter~3, Subsection~3.8]{BK}. Hence, $\{G_1(x),\,G_2(x)\mid x\in\Lambda\}$ is a Gaussian field. 

Finally, for each $x\in\Lambda$, we define $G(x):=G_1(x)+iG_2(x)$. Then  $\{G(x)\}_{x\in \Lambda}$ is a complex-valued Gaussian field. Formula \eqref{xrq5y3q2u} implies \eqref{cy6e64u43}. This, in turn, gives us the covariance and the pseudo-covariance of the Gaussian field $\{G(x)\}_{x\in \Lambda}$. 
\end{proof}

Let $\mathcal H=L^2(X,\sigma)$  and define an antilinear involution $J:\mathcal H\to \mathcal H$ by $(Jh)(x):=\overline{h(x)}$.
Let us consider a measurable mapping $x\mapsto L(x)\in\mathcal H$ defined $\sigma$-a.e.\ on $X$, and let $K(x,y):=(L(x),L(y))_{\mathcal H}$. We will now consider two examples of complex-valued Gaussian fields with covariance
$K(x,y)$.

\begin{example}\label{fyte6i47}
Let $\mathcal G=\mathcal H$, $\mathcal J=J$, and let $L_1(x)=L_2(x)=L(x)$. Obviously, conditions \eqref{ydxdrdxdrg}, \eqref{xrsa5yw4} are satisfied. Then, 
$$\Phi(x)=a^+(L(x))+a^-(L(x)),\quad \Psi(x)=a^+(JL(x))+a^-(JL(x)).$$
By Theorem~\ref{reaq5y43wu}, the corresponding Gaussian field $G(x)$ has the covariance $\mathcal K_1(x,y)=K(x,y)$ and the pseudo-covariance $\mathcal K_2(x,y)=(L(x),JL(y))_{\mathcal H}$. If $L(x)$ is a real-valued function for $\sigma$-a.a.\ $x\in X$, $\mathcal K_1(x,y)=\mathcal K_2(x,y)=K(x,y)$, while the function $K(x,y)$ is symmetric. Hence, as discussed in Introduction,  $R(x):=|G(x)|^2$ is a 2-permanental process, defined for $\sigma$-a.a.\ $x\in X$. If $L(x)$ is not real-valued on a set of positive $\sigma$ measure, then the moments of $R(x)$ 
are given by \eqref{dr6e6u43wq}, \eqref{d6e6ie4} with $\mathcal K_1(x,y)$, $\mathcal K_2(x,y)$ as above.
\end{example}

\begin{example}\label{vcrtw5y3}
Let $\mathcal G=\mathcal H\oplus\mathcal H$, $\mathcal J=J\oplus J$, and let $L_1(x)=(L(x),0)$, $L_2(x)=(0,L(x
))$. As easily seen, conditions \eqref{ydxdrdxdrg}, \eqref{xrsa5yw4} are satisfied.
We define, for each $h\in\mathcal H$, $a_1^+(h):=a^+(h,0)$, $a_2^+(h):=a^+(0,h)$ and similarly $a_1^-(h)$, $a_2^-(h)$. Then
$$\Phi(x)=a_1^+(L(x))+a_2^-(L(x)),\quad \Psi(x)=a_2^+(JL(x))+a_1^-(JL(x)).$$
For the corresponding Gaussian field $G(x)$, $\mathcal K_1(x,y)=K(x,y)$, while $\mathcal K_2(x,y)=0$, i.e., $G(x)$ is a proper Gaussian field. 
Hence, as discussed in Introduction,  $R(x):=|G(x)|^2$ is a permanental process,  defined for $\sigma$-a.a.\ $x\in X$. 
\end{example}

\begin{remark}
Let $G_1(x)$ and $G_2(x)$ be two independent copies of the Gaussian field from Example~\ref{fyte6i47}. Then, the Gaussian field $G(x)$ from Example~\ref{vcrtw5y3} can be constructed as $G(x)=\frac1{\sqrt2}\big(G_1(x)+iG_2(x)\big)$.
\end{remark}

The following example generalizes the constructions in Examples~\ref{fyte6i47} and \ref{vcrtw5y3}.

\begin{example}\label{xerwu56}
 Let $\mathcal G=\mathcal H\oplus\mathcal H$ be as in Example~\ref{vcrtw5y3} and consider a measurable mapping $x\mapsto(\alpha(x),\beta(x))\in\mathcal G^2$ defined $\sigma$-a.e.\ on $X$. Let
$$L_1(x):=\left(\frac{\alpha(x)+\beta(x)}2\,,\frac{\alpha(x)-\beta(x)}2\right),\quad L_2(x):=\left(\frac{\alpha(x)-\beta(x)}2\,,\frac{\alpha(x)+\beta(x)}2\right).$$
As easily seen, conditions \eqref{ydxdrdxdrg} and \eqref{xrsa5yw4} are satisfied. For the corresponding Gaussian field $G(x)$,
\begin{align}
\mathcal K_1(x,y)&=\frac12\big((\alpha(x),\alpha(y))_{\mathcal H}+(\beta(x),\beta(y))_{\mathcal H}\big)\notag,\\
\mathcal K_2(x,y)&=\frac12\big((\alpha(x), J\alpha(y))_{\mathcal H}-(\beta(x),J\beta(y))_{\mathcal H}\big).\notag
\end{align}
In the special case where $L(x)=\alpha(x)=\beta(x)$, this is just the construction from  Example~\ref{vcrtw5y3}. When choosing $\alpha(x)=\sqrt 2\, L(x)$ and $\beta(x)=0$, the corresponding Gaussian field $G(x)$ has the same finite-dimensional distributions as the Gaussian field from Example~\ref{fyte6i47}.
\end{example}

Let the conditions of Theorem~\ref{reaq5y43wu} be satisfied and $R(x)=|G(x)|^2$. To construct the Cox process $\Pi_R$ with correlation functions given by \eqref{xreew5u}, we further assume that, for each $\Delta\in\mathcal B_0(X)$, $\int_\Delta\mathbb E(R(x))\sigma(dx)<\infty$. By \eqref{xrsa5yw4} and \eqref{w909i8u7y689}, this is equivalent to the condition
\begin{equation}\label{vd5yw573}
\int_\Delta\|L_1(x)\|_{\mathcal G}^2\,\sigma(dx)=\int_\Delta\|L_2(x)\|_{\mathcal G}^2\,\sigma(dx)<\infty.\end{equation}

\setcounter{theorem}{2} 
\begin{example}[continued] Since $\mathcal G=\mathcal H$, we define  $L(x,y):=(L(x))(y)$. By \eqref{vd5yw573}, 
$$\int_{\Delta\times X} |L(x,y)|^2\sigma^{\otimes 2}(dx\,dy)<\infty.$$
Hence, for each $\Delta\in\mathcal B_0(X)$, we can define a Hilbert--Schmidt operator $L^\Delta$ in $\mathcal H$ with integral kernel $\chi_\Delta(x)L(x,y)$. Here $\chi_\Delta$ denotes the indicator function of the set $\Delta$. Define $K^\Delta:=L^\Delta(L^\Delta)^*$. This operator is nonnegative ($K^\Delta\ge0$),   trace-class, and has integral kernel $K^\Delta(x,y)=(L(x),L(y))_{\mathcal H}$ for $x,y\in\Delta$. (Note that $K^\Delta(x,y)$ vanishes outside $\Delta^2$). Thus, for $x,y\in\Delta$, the covariance $\mathcal K_1(x,y)$ of the Gaussian $G(x)$ is equal to $K^\Delta(x,y)$. 

Next, for a bounded linear operator $A\in\mathcal H$, we define the {\it transposed} of $A$ by $A^T:=JA^*J$. If $A$ is an integral operator with integral kernel $A(x,y)$, then $A^T$ is the integral operator with integral kernel $A^T(x,y)=A(y,x)$. Hence, for all $x,y\in\Delta$, the pseudo-covariance $\mathcal K_2(x,y)$ of the Gaussian $G(x)$ is equal to the integral kernel $Q^\Delta(x,y)$
of the operator $Q^\Delta:=L^\Delta(L^\Delta)^T$.

In the case where $L(x,y)$ is an integral kernel of a bounded linear operator $L$ in $\mathcal H$, we can define $K:=LL^*$ and $Q:=LL^T$, and $\mathcal K_1(x,y)=K(x,y)$, $\mathcal K_2(x,y)=Q(x,y)$, where $K(x,y)$ and $Q(x,y)$ are the integral kernels of the operators $K$ and $Q$, respectively.
\end{example}

\begin{example}[continued] We proceed similarly to Example~\ref{fyte6i47}. However, in this case, the moments of the Gaussian field $G(x)$ are  determined by the covariance $\mathcal K_1(x,y)$ only. Hence, if $L(x,y)$ is an integral kernel of a bounded linear operator $L$ in $\mathcal H$, the moments are determined by (the integral kernel of)  the operator $K:=LL^*$. Hence, without loss of generality, we may assume that $L=\sqrt K$, equivalently the operator $L$ is self-adjoint.
\end{example}

\setcounter{theorem}{5} 
\begin{example}[continued] Since $\mathcal G=\mathcal H$, we define  $\alpha(x,y):=(\alpha(x))(y)$ and similarly $\beta(x,y)$. 
In this case, condition \eqref{vd5yw573} means that, for each $\Delta\in\mathcal B_0(X)$,
$$\int_{\Delta\times X}(|\alpha(x,y)|^2+|\beta(x,y)|^2)\sigma^{\otimes 2}(dx\,dy)<\infty, $$
and we can proceed similarly to Example~\ref{fyte6i47}. Assume additionally that $\alpha(x,y)$ and $\beta(x,y)$ are integral kernels of operators $A,B\in\mathcal L(\mathcal H)$, respectively. Then the covariance $\mathcal K_1(x,y)$ of the Gaussian field $G(x)$ is the integral kernel of the operator $\frac12(AA^*+BB^*)$, while the pseudo-covariance $\mathcal K_2(x,y)$ is the integral kernel of the operator $\frac12(AA^T-BB^T)$.
\end{example}

\section{Quasi-free states on the CCR algebra}\label{xrwaq4q2}
In this section, we assume that  $\mathcal H$ is a separable complex Hilbert space with an antilinear involution $J$ in $\mathcal H$ 
satisfying $(Jf,Jh)_{\mathcal H}=(h,f)_{\mathcal H}$ for all $f,h\in\mathcal H$. 
Let $\mathcal V$ be a dense subspace of $\mathcal H$ that is invariant for $J$. 
Let $\mathfrak F$ be a separable Hilbert space and  $\mathfrak D$ a dense subspace of $\mathfrak F$. For each $h\in\mathcal V$, let $A^+(h),\, A^-(h):\mathfrak D\to\mathfrak D$  be linear operators  satisfying  the following assumptions: 
\begin{itemize}
\item[(i)] $A^+(h)$ and $A^-(h)$ depend linearly on $h\in\mathcal V$; 

\item[(ii)] the domain of the adjoint operator of $A^+(h)$ in $\mathfrak F$ contains $\mathfrak D$ and $A^+(h)^*\restriction\mathfrak D=A^-(Jf)$; 

\item[(iii)] the operators $A^+(h)$, $A^-(h)$ satisfy the CCR: 
\begin{equation}\label{ydey7e76}
[A^+(f),A^+(h)]=[A^-(f),A^-(h)]=0,\quad [A^-(f),A^+(h)]=(h,Jf)_{\mathcal H}\end{equation}
for all $f,h\in\mathcal V$.
  \end{itemize}

Let $\mathbb A$ denote the complex unital $*$-algebra generated by the operators $A^+(h)$, $A^-(h)$ ($h\in\mathcal V$). Let $\tau:\mathbb A\to\mathbb C$ be a state on $\mathbb A$, i.e., $\tau$ is linear, $\tau(\mathbf 1)=1$ and $\tau(a^*a)\ge0$ for all $a\in\mathbb A$. For each $h\in\mathcal V$, we define a Hermitian operator 
\begin{equation}\label{e5w7w37w2} B(h):=A^+(h)+A^-(Jh).
\end{equation}
These operators satisfy the commutation relation
\begin{equation}\label{vctsdtu6}
[B(f),B(h)]=2\Im(h,f)_{\mathcal H},\quad h,f\in\mathcal H.
\end{equation}
 Note that 
 $$A^+(h)=\frac12(B(h)-iB(ih)),\quad A^-(h)=\frac12(B(Jh)+iB(Jh)).$$
  Therefore, we can think of the algebra $\mathbb A$ as generated by the operators $B(h)$ ($h\in\mathcal V$), subject to the commutation relation \eqref{vctsdtu6}. Hence, the state $\tau$ is completely determined by the functionals $T^{(n)}:\mathcal V^n\to\mathbb C$  ($n\ge1$), where $T^{(1)}(h):=\tau(B(h))$ and
$$ T^{(n)}(h_1,\dots,h_n):=\tau\big((B(h_1)-T^{(1)}(h_1))\dotsm (B(h_n)-T(h_n))\big),\quad n\ge2.$$

The state $\tau$ is called {\it quasi-free} if, for each $n\in\mathbb N$, $T^{(2n+1)}=0$  and
$$T^{(2n)}(h_1,\dots,h_{2n})=\sum T^{(2)}(h_{i_1},h_{j_1})\dotsm T^{(2)}(h_{i_n},h_{j_n}) $$
where the summation is over all partitions $\{i_1,j_1\},\dots,\{i_n,j_n\}$ of $\{1,\dots,2n\}$ with $i_k<j_k$ ($k=1,\dots,n$), see e.g.\  \cite[Section~5.2]{BR}.

\begin{remark} \label{vtw53u}
Let $\phi:\mathcal V\to\mathbb C$ be a linear functional. 
For each $h\in\mathcal V$, we define operators $\mathbf A^+(h):=A^+(h)+\phi'(h)$ and  $\mathbf A^-(h):=A^-(h)+\phi(h)$, where $\phi'(h):=\overline{\phi(Jh)}$. The operators $\mathbf A^+(h)$, $\mathbf A^-(h)$ also satisfy the  conditions (i)--(iii) discussed above. Obviously, the algebra generated by the operators $\mathbf A^+(h)$, $\mathbf A^-(h)$ coincides with $\mathbb A$.
 If $\tau:\mathbb A\to\mathbb C$ is a quasi-free state with respect to the operators $A^+(h)$, $A^-(h)$, then  $\tau$ is also quasi-free with respect to the operators $\mathbf A^+(h)$, $\mathbf A^-(h)$.
\end{remark}

Let us now present an explicit construction of a representation of the CCR algebra $\mathbb A$ and a quasi-free state $\tau$ on it. This construction resembles the {\it Bogoliubov transformation}, see e.g.\ \cite[Subsection 5.2.2.2]{BR} or \cite[Section~4]{Berezin}\footnote{In \cite{Berezin}, a Bogoliubov transformation is called a {\it linear canonical transformation}.}.

Let $\mathcal E$ be a separable Hilbert space with an antilinear involution $\mathcal J$ satisfying $(\mathcal Jf,\mathcal Jg)_{\mathcal E}=(g,f)_{\mathcal E}$ for all $f,g\in\mathcal E$. Let $K_i\in\mathcal L(\mathcal H,\mathcal E)$ ($i=1,2$). Denote 
\begin{equation}\label{rdtsw5u4w3wq}K_i':=\mathcal JK_iJ\in\mathcal L(\mathcal H,\mathcal E)
\end{equation}
 and assume that
\begin{align}
(K'_2)^*K_1-(K_1')^*K_2&=0,\notag\\
K_2^*K_2-K_1^*K_1=1.\label{cdtwe64u53e}
\end{align}
For each $h\in\mathcal H$, we define, in the symmetric Fock space $\mathcal F(\mathcal E)$, the operators
\begin{equation}\label{tsw53w3}
A^+(h):=a^+(K_2h)+a^-(K_1h),\quad A^-(h):=a^-(K_2'h)+a^+(K_1'h)
\end{equation}
with domain $\mathcal F_{\mathrm{fin}}(\mathcal E)$. Here $a^+(\cdot)$ and $a^-(\cdot)$ are the creation and annihilation operators in  $\mathcal F(\mathcal E)$, respectively.
It follows from \eqref{tsw654w}, \eqref{cdtwe64u53e}, and \eqref{tsw53w3} that $A^+(h)$ and $A^-(h)$ satisfy the conditions (i)--(iii)   with 
$\mathcal V=\mathcal H$, $\mathfrak F=\mathcal F(\mathcal E)$, and $\mathfrak D=\mathcal F_{\mathrm{fin}}(\mathcal E)$. 

Let $\mathbb A$ denote the corresponding CCR algebra and let $\tau:
\mathbb A\to\mathbb C$ be the {\it vacuum state on $\mathbb A$}, i.e., $\tau(a)$$:=(a\Omega,\Omega)_{\mathcal F(\mathcal E)}$\,. For each $h\in\mathcal H$,
\begin{equation}\label{vcreaw53y}
B(h)=A^+(h)+A^-(Jh)=a^+((K_2+\mathcal JK_1)h)+a^-((K_1+\mathcal JK_2)h). \end{equation}	
In particular, $\tau(B(h))=0$. Hence, it easily follows from \eqref{vcreaw53y} that $\tau$ is a quasi-free state with 
\begin{equation}\label{rtwe64ue3}
T^{(2)} (f,h)=\big((K_1+\mathcal JK_2)f,(K_1+\mathcal JK_2)h\big)_{\mathcal E}\,.
\end{equation}

\begin{remark}Note that, in the classical Bogoliubov transformation, one chooses $\mathcal E=\mathcal H$. 
\end{remark}

\begin{remark}\label{vrtw52}
Choosing $\mathcal E=\mathcal H$, $K_1=1$ and $K_2=0$, we get $A^+(h)=a^+(h)$, $A^-(h)=a^-(h)$. In this case, the vacuum state is quasi-free,  with $T^{(1)}=0$ and $T^{(2)}(f,h)=(f,Jh)_{\mathcal H}$.
\end{remark}

\section{Particle density and correlation functions}\label{tew532w}

Let $\mathfrak F$ be a separable Hilbert space and let $\mathfrak D$ be a dense subspace of $\mathfrak F$.  For each $\Delta\in\mathcal B_0(X)$, let  $\rho(\Delta):\mathfrak D\to\mathfrak D$ be a linear Hermitian operator in $\mathfrak F$. We further assume that the operators $\rho(\Delta)$ commute, i.e.,  
$[\rho(\Delta_1),\rho(\Delta_2)]=0$ and for any $\Delta_1,\Delta_2\in\mathcal B_0(X)$.
Let $\mathcal A$ denote 
 the complex unital (commutative) $*$-algebra generated by $\rho(\Delta)$ ($\Delta\in\mathcal B_0(X)$). Let $\Omega$ be a fixed vector in $\mathfrak F$  with $\|\Omega\|_{\mathfrak F}=1$, and let a state $\tau:\mathcal A\to\mathbb C$ be defined by $\tau(a):=(a\Omega,\Omega)_{\mathfrak F}$ for $a\in\mathcal A$.

We define {\it Wick polynomials in $\mathcal A$} by the following recurrence formula:
\begin{align}
{:}\rho(\Delta){:}&=\rho(\Delta),\notag\\
{:}\rho(\Delta_1)\dotsm \rho(\Delta_{n+1}){:}&=\rho(\Delta_{n+1})\,
{:}\rho(\Delta_1)\dotsm \rho(\Delta_{n}){:}\notag\\
&\quad-\sum_{i=1}^n
{:}\rho(\Delta_1)\dotsm \rho(\Delta_{i-1})\rho(\Delta_i\cap\Delta_{n+1})\rho(\Delta_{i+1})\dotsm \rho(\Delta_n){:}\,,\label{dsaea78}
 \end{align}
 where $\Delta,\Delta_1,\dots,\Delta_{n+1}\in\mathcal B_0(X)$ and $n\in\mathbb N$. 
 It is easy to see that, for each permutation $\pi\in \mathfrak S_n$,
$${:}\rho(\Delta_1)\cdots \rho(\Delta_n){:} = {:}\rho(\Delta_{\pi(1)})\cdots \rho(\Delta_{\pi(n)}){:}\, .$$

We assume that, for each $n\in\mathbb N$, there exists a symmetric measure $\theta^{(n)}$ on $X^n$ that is concentrated on $X^{(n)}$ (i.e., $\theta^{(n)}(X^n\setminus X^{(n)})=0$) and satisfies\footnote{In view of formulas \eqref{dsaea78}, \eqref{6esuw6u61d}, it is natural to call $\theta^{(n)}$ the {\it $n$-th correlation measure of the family of operators} $(\rho(\Delta))_{\Delta\in\mathcal B_0(X)}$.}
\begin{equation}\label{6esuw6u61d}
\theta^{(n)}\big(\Delta_1\times\dots\times\Delta_n\big)=\frac1{n!}\,
\tau\big({:}\rho(\Delta_1)\dotsm \rho(\Delta_{n}){:}\big),\quad  \Delta_1,\dots,\Delta_{n}\in\mathcal B_0(X).\end{equation}
Furthermore, we assume that, for each $\Delta\in\mathcal B_0(X)$, there exists a constant $C_\Delta>0$ such that
\begin{equation}\label{fte76i4}
\theta^{(n)}(\Delta^n)\le  C_\Delta^n,\quad n\in\mathbb N,
\end{equation}
and for any sequence
    $\{\Delta_{l}\}_{l\in\mathbb{N}}\subset\mathcal{B}_{0}(X)$ such
    that $\Delta_{l}\downarrow\varnothing$ (i.e.,     $\Delta_1\supset\Delta_2\supset\Delta_3\supset\cdots$ and $\bigcap_{l=1}^\infty\Delta_l=\varnothing$), we have $C_{\Delta_{l}}\rightarrow
    0$ as $l\rightarrow\infty$.

\begin{theorem}[\!\!\cite{LM}] {\rm (i)} Under the above assumptions, there exists a unique point process $\mu$ in $X$ whose correlation measures are $(\theta^{(n)})_{n=1}^\infty$. 

{\rm (ii)} Let $\mathfrak D':=\{a\Omega\mid a\in\mathcal A\}$ and let $\mathfrak F'$ denote the closure of $\mathfrak D'$ in $\mathfrak F$. Then  each operator $(\rho(\Delta),\mathfrak D')$ is essentially self-adjoint in $\mathfrak F'$, and the operator-valued measures of the closures of the operators  $(\rho(\Delta),\mathfrak D')$ commute.  Furthermore, there exists a unique unitary operator $U:\mathfrak F'\to L^2(\Gamma(X),\mu)$ satisfying $U\Omega=1$ and 
\begin{equation}\label{te5w6u3e}
U(\rho(\Delta_1)\dotsm \rho(\Delta_{n})\Omega)=\gamma(\Delta_1)\dotsm\gamma(\Delta_n)\end{equation}
 for any $\Delta_1,\dots,\Delta_{n}\in\mathcal B_0(X)$ ($n\in\mathbb N$). In particular,
\begin{equation}\label{cts6wu4w5}
\tau\big(\rho(\Delta_1)\dotsm \rho(\Delta_{n})\big)= \int_{\Gamma(X)}\gamma(\Delta_1)\dotsm\gamma(\Delta_n)\,\mu(d\gamma).\end{equation}
\label{tes56uwe4u6}
\end{theorem}

We finish this section with a formal observation. 
Let again $\mathcal H=L^2(X,\sigma)$ and the antilinear involution $J$ in $\mathcal H$ be given by  $(Jf)(x):=\overline{f(x)}$. Let $A^+(h)$ and $A^-(h)$ ($h\in\mathcal V$) be  operators satisfying the CCR, and let the corresponding   operators $A^+(x)$, $A^-(x)$ ($x\in X$) be derfined by \eqref{ray45}. For each $\Delta\in\mathcal B_0(X)$, let $\rho(\Delta)$ be given by \eqref{raq5wu}.
The CCR \eqref{ydey7e76} and formulas \eqref{raq5wu}, \eqref{dsaea78} imply that, for any $\Delta_1,\dots,\Delta_{n}\in\mathcal B_0(X)$,
\begin{equation}\label{terw5yw35}
{:}\rho(\Delta_1)\dotsm \rho(\Delta_{n}){:}=\int_{\Delta_1\times\dots\times \Delta_n}A^+(x_n)\dotsm A^+(x_1)A^-(x_1)\dotsm A^-(x_n)\sigma^{\otimes n}(dx_1\dotsm dx_n).\end{equation}
Thus, the Wick polynomials correspond to the Wick normal ordering, in which all the creation operators are to the left of  all the annihilation operators. 
Hence, by \eqref{6esuw6u61d} and~\eqref{terw5yw35}, we formally obtain
\begin{align*}
&\theta^{(n)}\big(\Delta_1\times\dots\times\Delta_n\big)\notag\\
&\quad =\frac1{n!}\int_{\Delta_1\times\dots\times \Delta_n}\tau\big(A^+(x_n)\dotsm A^+(x_1)A^-(x_1)\dotsm A^-(x_n)\big)\sigma^{\otimes n}(dx_1\dotsm dx_n).\end{align*}
Therefore, by \eqref{5w738}, the point process $\mu$ from Theorem~\ref{tes56uwe4u6} has the correlation functions
$$k^{(n)}(x_1,\dots,x_n)=\tau\big(A^+(x_n)\dotsm A^+(x_1)A^-(x_1)\dotsm A^-(x_n)\big).$$

Below we will see that, in the case of a Cox process $\Pi_R$, where  $R(x)=|G(x)|^2$ and $G(x)$ is a complex-valued Gaussian field from Section~\ref{ew56u3wu}, the  above formal calculations can be given a rigorous meaning.  We will start, however, with the simpler case of a Poisson point process.

\section{Application of Theorem~\ref{tes56uwe4u6} to Poisson point processes}\label{xeeraq54q}

Recall Remark~\ref{vtw53u}. Let  $\mathcal V$ denote the (dense) subspace of $\mathcal H=L^2(X,\sigma)$ consisting of all measurable bounded (versions of) functions $h:X:\to\mathbb C$ with compact support. 
Let us fix a function $\lambda\in L^2_{\mathrm{loc}}(X,\sigma)$ and define a functional 
$\phi:\mathcal V\to\mathbb C$ by 
\begin{equation}\label{s5w5as}
\phi(h):=\int_X h(x)\lambda(x)\sigma(dx).
\end{equation}
 Note that 
 \begin{equation}\label{xse5w64ue}
 \phi'(h):=\int_X h(x)\overline{\lambda(x)}\,\sigma(dx).
 \end{equation} 
For each  $h\in\mathcal V$, we define operators $A^+(h),A^-(h)\in\mathcal L(\mathcal F_{\mathrm{fin}}(\mathcal H))$ by 
\begin{equation}\label{er4q53}
A^+(h):=a^+(h)+\phi'(h),\quad A^-(h):=a^-(h)+\phi(h).\end{equation}
 Here $a^+(h)$ and $a^-(h)$ are the  creation and annihilation operators in $\mathcal F(\mathcal H)$. By Remarks~\ref{vtw53u} and~\ref{vrtw52}, the vacuum state $\tau$ on the CCR algebra generated by $A^+(h)$, $A^-(h)$ ($h\in\mathcal V$) is quasi-free.

Let $a^+(x)$ and $a^-(x)$ be the operator-valued distributions corresponding  to $a^+(h)$ and $a^-(h)$, respectively. Then $A^+(x)=a^+(x)+\overline{\lambda(x)}$ and $A^-(x)=a^-(x)+\lambda(x)$.  Hence, the corresponding particle density takes the form 
$$\rho(x)=\lambda(x)a^+(x)+\overline{\lambda(x)}\,a^-(x)+a^+(x)a^-(x)+|\lambda(x)|^2.$$
Our next aim is to rigorously define, for each $\Delta\in\mathcal B_0(X)$, an operator $\rho(\Delta)=\int_\Delta A^+(x)A^-(x)\sigma(dx)\in\mathcal L(\mathcal F_{\mathrm{fin}}(\mathcal H))$.

 We clearly have, for each $\Delta\in\mathcal B_0(X)$,
$$\int_\Delta \lambda(x)a^+(x)\,\sigma(dx)=\int_X\chi_\Delta(x)\lambda(x)a^+(x)\,\sigma(dx)=a^+(\chi_\Delta \lambda) $$
and similarly
$$\int_\Delta \overline{\lambda(x)}\,a^-(x)\,\sigma(dx)=a^-(\chi_\Delta \overline{\lambda}).$$
(Note that $\chi_\Delta \lambda\in\mathcal H$.) Next, for $h\in\mathcal V$ and 
$f^{(n)}\in\mathcal F_n(\mathcal H)$,
$$\big(a^-(h)f^{(n)}\big)(x_1,\dots,x_{n-1})=n\int_X h(x)f^{(n)}(x,x_1,\dots,x_{n-1})\sigma(dx).$$
Hence
\begin{equation}\label{cyd64uew6wqa}
(a^-(x)f^{(n)})(x_1,\dots,x_{n-1})=nf^{(n)}(x,x_1,\dots,x_{n-1}),\end{equation}
which implies
\begin{equation}\label{w4q53aestyp}
\bigg(\int_\Delta a^+(x)a^-(x)\sigma(dx)f^{(n)}\bigg)(x_1,\dots,x_n)=\big(\chi_\Delta(x_1)+\dots+\chi_\Delta(x_n)\big)f^{(n)}(x_1,\dots,x_n).
\end{equation}
Hence, $a^0(\chi_\Delta):=\int_\Delta a^+(x)a^-(x)\sigma(dx)\in\mathcal L(\mathcal F_{\mathrm{fin}}(\mathcal H))$. The $a^0(\chi_\Delta)$ is called a {\it neutral operator}. 

Thus, for each $\Delta\in\mathcal B_0(X)$, we have rigorously defined
\begin{equation}\label{rtw52qy}
\rho(\Delta)=a^+(\chi_\Delta\lambda)+a^-(\chi_\Delta\overline{\lambda})+a^0(\chi_\Delta)+\int_\Delta|\lambda(x)|^2\sigma(dx)\in\mathcal L(\mathcal F_{\mathrm{fin}}(\mathcal H)).\end{equation}
Obviously, $\rho(\Delta)$ is a Hermitian operator in $\mathcal F(\mathcal H)$. To construct a state on the corresponding $*$-algebra, we use the vacuum $\Omega$
in the Fock space $\mathcal F(\mathcal H)$.

\begin{proposition}\label{t7re6}
The operators $(\rho(\Delta))_{\Delta\in\mathcal B_0(X)}$ defined by \eqref{rtw52qy} and the vacuum state $\tau$ satisfy the assumptions of Theorem~\ref{tes56uwe4u6}. In this case, $\theta^{(n)}=\frac1{n!}(|\lambda|^2\sigma)^{\otimes n}$, so that $\mu$ is the Poisson point process with intensity  $|\lambda(x)|^2$. Furthermore, we have $\mathfrak F'=\mathcal F(\mathcal H)$.
\end{proposition}

\begin{remark}
For the Poisson point process $\mu$ with intensity  $|\lambda|^2$,  the existence of the unitary isomorphism $U:\mathcal F(\mathcal H)\to L^2(\Gamma(X),\mu)$ that satisfies \eqref{te5w6u3e}, \eqref{cts6wu4w5} is a well-known fact, see e.g.\ \cite{Surgailis}. Our approach to the construction of the isomorphism $U$ may be compared with paper \cite{GGPS}. 
\end{remark}

\begin{proof}[Proof of Proposition \ref{t7re6}] For any $\Delta_1,\Delta_2\in\mathcal B_0(X)$, the commutation $[\rho(\Delta_1),\rho(\Delta_2)]=0$ follows from the CCR  and the commutation relations
$$\big[a^0(\chi_{\Delta_1}),a^+(\chi_{\Delta_2}\lambda)\big]=a^+(\chi_{\Delta_1\cap\Delta_2}\lambda),\quad \big[a^0(\chi_{\Delta_1}),a^-(\chi_{\Delta_2}\overline{\lambda})\big]=-a^-(\chi_{\Delta_1\cap\Delta_2}\overline{\lambda}).$$

Next, let $C\in\mathcal L(\mathcal F_{\mathrm{fin}}(\mathcal H))$. Similarly to \eqref{cyd64uew6wqa}, \eqref{w4q53aestyp}, we see that, for each $\Delta\in\mathcal B_0(X)$,  $\int_\Delta a^+(x)Ca^-(x)\sigma(dx) $  determines an operator from $\mathcal L(\mathcal F_{\mathrm{fin}}(\mathcal H))$. In particular, for $f\in\mathcal H$ and $n\in\mathbb N$,
\begin{equation}\label{dr6wu3}
\int_\Delta a^+(x)Ca^-(x)\sigma(dx)f^{\otimes n}=n(\chi_\Delta f)\odot(Cf^{\otimes(n-1)}).\end{equation}
Therefore,
\begin{align*}
&\int_\Delta A^+(x)CA^-(x)\sigma(dx)\\
&\quad=\int_\Delta a^+(x)Ca^-(x)\sigma(dx)+a^+(\chi_\Delta\lambda)C+Ca^-(\chi_\Delta\overline{\lambda})+\int_\Delta|\lambda(x)|^2\sigma(dx)\,C\in\mathcal L(\mathcal F_{\mathrm{fin}}(\mathcal H)).
\end{align*}
Hence, we may define, for any $\Delta_1,\dots,\Delta_n\in\mathcal B_0(X)$,
\begin{align*}
&\int_{\Delta_1\times\dots\times \Delta_n}A^+(x_n)\dotsm A^+(x_1)A^-(x_1)\dotsm A^-(x_n)\,\sigma^{\otimes n}(dx_1\dotsm dx_n)\\
&\quad:=\int_{\Delta_n}A^+(x_n)\bigg(\int_{\Delta_{n-1}}A^+(x_{n-1})\bigg(\dotsm\int_{\Delta_1}A^+(x_1)A^-(x_1)\sigma(dx_1)\bigg)\\
&\qquad\dotsm  A^-(x_{n-1})\sigma(dx_{n-1})\bigg)A^-(x_n)\sigma(dx_n)\in\mathcal L(\mathcal F_{\mathrm{fin}}(\mathcal H)).\end{align*}

We next state that formula \eqref{terw5yw35} now holds rigorously. Indeed, a direct calculation shows that, for any $\Delta_1,\Delta_2 \in\mathcal B_0(X)$ and $C\in\mathcal L(\mathcal F_{\mathrm{fin}}(\mathcal H))$,
\begin{align}
&\rho(\Delta_1)\int_{\Delta_2}A^+(x)CA^-(x)\,\sigma(dx)\notag\\
&\quad=
\int_{\Delta_2}A^+(x)\rho(\Delta_1)CA^-(x)\,\sigma(dx)+\int_{\Delta_1\cap\Delta_2}A^+(x)CA^-(x)\,\sigma(dx).
\label{cx4t3y7}\end{align}
Now formula \eqref{terw5yw35} follows by induction from \eqref{dsaea78} and \eqref{cx4t3y7}.

 Applying the vacuum state $\tau$ to \eqref{terw5yw35} , we get
\begin{equation}\label{ctea5ywq}
\tau\big({:}\rho(\Delta_1)\dotsm \rho(\Delta_{n}){:}\big)=\int_{\Delta_1\times\dots\times\Delta_n}|\lambda(x_1)|^2\dotsm|\lambda(x_n)|^2\sigma^{\otimes n}(dx_1\dotsm dx_n). \end{equation}
Since the measure $\sigma$ is non-atomic, $\sigma^{\otimes n}$ is concentrated on $X^{(n)}$. By \eqref{6esuw6u61d} and \eqref{ctea5ywq}, 
estimate \eqref{fte76i4} holds with $C_\Delta=\int_\Delta |\lambda(x)|^2\sigma(dx)$. Hence, the assumptions of Theorem~\ref{tes56uwe4u6} are satisfied. The form of the correlation measures implies that $\mu$ is the Poisson point process with intensity $|\lambda(x)|^2$.

Finally, the proof of the equality $\mathfrak F'=\mathcal F(\mathcal H)$ is standard and we leave it to the interested reader. \end{proof}

\section{Application of Theorem~\ref{tes56uwe4u6} to hafnian point processes}\label{yd6w6wdd}

We will use below the notations from Section~\ref{ew56u3wu}. We assume that conditions  \eqref{ydxdrdxdrg}, \eqref{xrsa5yw4}, and \eqref{vd5yw573} are satisfied. Let also the subspace $\mathcal V$ of $\mathcal H$ be as in Section~\ref{xeeraq54q}.

Let $h\in\mathcal V$. By the Cauchy inequality,
$$\int_X |h(x)|\,\|L_i(x)\|_{\mathcal G}\,\sigma(dx)<\infty,\quad i=1,2.$$
Hence, by using e.g.\ \cite[Chapter 10, Theorem~3.1]{BUS},
we can define
\begin{equation*}
\int_X hL_i\,d\sigma, \int_X h\,\mathcal JL_i\,d\sigma\in\mathcal G\end{equation*}
as  Bochner integrals. 

\begin{example}
Recall Examples \ref{fyte6i47} and \ref{vcrtw5y3}.  As easily seen, for each $h\in\mathcal V$,
\begin{equation}\label{eaw4ayaqwy}
\int_X hL\,d\sigma =(L^\Delta)^Th,\ \int_X h\, JL\,d\sigma =(L^\Delta)^*h\in\mathcal H,
\end{equation}
where $\Delta\in\mathcal B_0(X)$ is chosen so that $h$ vanishes outside $\Delta$. In particular, if $L(x,y)$ is the integral kernel of an operator $L\in\mathcal L(\mathcal H)$, then we can replace $L^\Delta$ in formula \eqref{eaw4ayaqwy} with $L$. Furthermore, in the latter case, we could set $\mathcal V=\mathcal H$.
\end{example}

Denote $\mathcal E:=\mathcal H\oplus \mathcal G$. We recall the well-known unitary isomorphism between\linebreak  $\mathcal F(\mathcal H)\otimes\mathcal F(\mathcal G)$ and $\mathcal F(\mathcal E)$. In view of our considerations in Sections~\ref{ew56u3wu} and \ref{xeeraq54q}, see, in particular, formulas \eqref{waaq4yq5y}, \eqref{ytd7r57}, and \eqref{s5w5as}--\eqref{er4q53},
we consider in $\mathcal F(\mathcal E)$  the following linear operators with domain $\mathcal F_{\mathrm{fin}}(\mathcal E)$,
\begin{align}
A^+(h):=&a^+\bigg(h,\int_X h\,\mathcal JL_2\,d\sigma\bigg)+a^-\bigg(0,\int_X h\,\mathcal JL_1\,d\sigma\bigg),\notag\\
A^-(h):=&a^+\bigg(0,\int_X h L_1\,d\sigma\bigg)+a^-\bigg(h,\int_X hL_2\,d\sigma\bigg),\quad h\in\mathcal V.\label{tdr65eS}
\end{align}

\begin{proposition}
The operators $A^+(h)$, $A^-(h)$  defined by \eqref{tdr65eS} satisfy the conditions (i)--(iii) from Section~\ref{xrwaq4q2} with $\mathfrak F=\mathcal F(\mathcal E)=\mathcal F(H\oplus \mathcal G)$ and $\mathfrak D=\mathcal F_{\mathrm{fin}}(\mathcal E)$.  The vacuum state on the corresponding CCR algebra  is quasi-free with $T^{(1)}=0$ and
\begin{align}
T^{(2)}(f,h)&=\int_X \overline{f(x)}\, h(x)\sigma(dx)\notag\\
&\quad+2\int_{X^2}\Re\left(f(x)h(y)\overline{\mathcal K_2(x,y)}+\overline{f(x)}\,h(y)\mathcal K_1(x,y)\right)\sigma^{\otimes 2}(dx\,dy).
\end{align}
Here $\mathcal K_1$ and $\mathcal K_2$ are defined by \eqref{w909i8u7y689} and \eqref{ftsqw43qd}, respectively. 
\end{proposition}

\begin{proof} The first statement of the proposition is obvious in view of the commutation of the operators $\Phi(x)$, $\Psi(x)$. By \eqref{e5w7w37w2} and \eqref{tdr65eS},
\begin{align*}
B(h)&=a^+\bigg(h,\int_X (h\,\mathcal JL_2+(Jh)\, L_1)d\sigma\bigg)+a^-\bigg(Jh, \int_X(h\,\mathcal JL_1+(Jh)\, L_2)d\sigma\bigg).
\end{align*}
Hence, by \eqref{ydxdrdxdrg}, \eqref{xrsa5yw4}, \eqref{w909i8u7y689}, and \eqref{ftsqw43qd},  the second statement of the proposition also follows.
\end{proof}

\begin{remark}\label{yre7i645i7r45e}
 Assume that, for $i=1,2$, the map $\mathcal V\ni h\mapsto \int_X hL_i\,d\sigma\in\mathcal G$ extends by continuity to a bounded linear operator $\mathbb L_i\in\mathcal L(\mathcal H,\mathcal G)$. Then, by \eqref{rdtsw5u4w3wq} and \eqref{tdr65eS},
$$A^+(h)=a^+(h,\mathbb L_2'h)+a^-(0,\mathbb L_1'h),\quad A^-(h)=a^+(0,\mathbb L_1h)+a^-(h,\mathbb L_2h).$$
This quasi-free representation of the CCR is a special case of \eqref{cdtwe64u53e}, \eqref{tsw53w3}. By \eqref{rtwe64ue3}, 
$$ T^{(2)}(f,h)=(h,f)_{\mathcal H}+\big((\mathcal J\mathbb L_1+\mathbb L_2)Jf,(\mathcal J\mathbb L_1+\mathbb L_2)Jh\big)_{\mathcal G}.$$
\end{remark}

By \eqref{tdr65eS}, the corresponding operator-valued distributions $A^+(x)$ and $A^-(x)$ are given by
\begin{align}
A^+(x)&=a_1^+(x)+a_2^+(\mathcal JL_2(x))+a_2^-(\mathcal JL_1(x)),\notag\\
A^-(x)&=a_1^-(x)+a_2^+(L_1(x))+a_2^-(L_2(x)), \label{ytyuryr}\end{align}
compare with \eqref{waaq4yq5y}, \eqref{ytd7r57}. The operators $a_i^\pm(\cdot)$ ($i=1,2$) are defined similarly to Example~\ref{vcrtw5y3}.

\begin{proposition}\label{cyse6ui5i} Let $A^+(x)$, $A^-(x)$ be given by \eqref{ytyuryr}, and let $C\in\mathcal L(\mathcal F_{\mathrm{fin}}(\mathcal E))$. For each $\Delta\in\mathcal B_0(X)$,  $\int_\Delta A^+(x)CA^-(x)\,\sigma(dx) $  determines an operator from $\mathcal L(\mathcal F_{\mathrm{fin}}(\mathcal E))$, in the sense explained in the proof.  \end{proposition}

\begin{proof} It is sufficient to prove the statement when $C\in\mathcal L(\mathcal F_n(\mathcal E),\mathcal F_m(\mathcal E))$. We also fix $\Delta\in\mathcal B_0(X)$.
By \eqref{cr6sw6u4e} and \eqref{vd5yw573}, 
\begin{align}
&\int_\Delta \|a_2^-(\mathcal JL_1(x))Ca_2^+(L_1(x))\|_{\mathcal L(\mathcal F_{n-1}(\mathcal E),\mathcal F_{m-1}(\mathcal E))}\,\sigma(dx)\notag\\
&\quad\le\|C\|_{\mathcal L(\mathcal F_{n}(\mathcal E),\mathcal F_{m}(\mathcal E))}\sqrt{nm}\,\int_{\Delta}\|L_1(x)\|_{\mathcal G}^2 \,\sigma(dx)<\infty.\notag\end{align}
Hence, by \cite[Chapter 10, Theorem~3.1]{BUS}, the following Bochner integral is well-defined: 
\begin{equation*}
\int_\Delta a_2^-(\mathcal JL_1(x))C a_2^+(L_1(x))\sigma(dx)\in \mathcal L(\mathcal F_{n-1}(\mathcal E),\mathcal F_{m-1}(\mathcal E)).\end{equation*}
Note that, by e.g.\ \cite[Chapter 10, Theorem~3.2]{BUS}, for each  $f^{(n-1)}\in\mathcal F_{n-1}(\mathcal E)$,
$$\bigg(\int_\Delta a_2^-(\mathcal JL_1(x))C a_2^+(L_1(x))\sigma(dx)\bigg)f^{(n-1)}=\int_\Delta a_2^-(\mathcal JL_1(x))C a_2^+(L_1(x))f^{(n-1)}\sigma(dx),$$
where the right hand side is a Bochner integral with values in $\mathcal F_{m-1}(\mathcal E)$. 
 The proof of existence of the other Bochner integrals of the type
 $\int_\Delta a_2^\pm(\mathcal JL_i(x))Ca_2^\pm(L_j(x))\sigma(dx)$ ($i,j\in\{1,2\}$) is similar.
 
  Next, we define
$$\int_\Delta a_1^+(x)Ca_1^-(x)\sigma(dx)\in\mathcal L(\mathcal F_{n+1}(\mathcal E),\mathcal F_{m+1}(\mathcal E))$$
by analogy \eqref{dr6wu3}. 
Let $(e_i)_{i\ge1}$ be an orthonormal basis in $\mathcal H$ such that $Je_i=e_i$ for all $i\ge1$. As easily seen,
\begin{equation}\label{r5w5w32}
  \int_\Delta a_1^+(x)Ca_1^-(x)\sigma(dx)=\sum_{i,j\ge1}(\chi_\Delta e_i,e_j)_{\mathcal H}\,a_1^+(e_j)Ca_1^-(e_i),\end{equation}
where the series converges strongly in $\mathcal L(\mathcal F_{n+1}(\mathcal E),\mathcal F_{m+1}(\mathcal E))$.

By \eqref{vd5yw573},  we can define a linear operator $L_2^\Delta\in\mathcal L(\mathcal G,\mathcal H)$ by
$$(L_2^\Delta g)(x):=\chi_\Delta(x)(L_2(x),\mathcal Jg)_{\mathcal G}\,.$$
By analogy with \eqref{dr6wu3}, we define
\begin{equation}\label{cw5uu5}
\int_\Delta a_1^+(x)Ca_2^-(L_2(x))\sigma(dx)\in\mathcal L(\mathcal F_{n+1}(\mathcal E),\mathcal F_{m+1}(\mathcal E))\end{equation}
that satisfies, for each $f=(h,g)\in\mathcal E$,
\begin{equation}\label{vctsw5w3}
\int_\Delta a_1^+(x)Ca_2^-(L_2(x))\sigma(dx)f^{\otimes (n+1)}=(n+1)\big(L_2^\Delta g\big)\odot(Cf^{\otimes n}).\end{equation}
Let $(u_j)_{j\ge1}$ be an orthonormal basis in $\mathcal G$ such that $\mathcal Ju_j=u_j$ for all $j\ge1$. Then, similarly to \eqref{r5w5w32}, we obtain
\begin{equation}\label{fys6w6u4}
\int_\Delta a_1^+(x)Ca_2^-(L_2(x))\sigma(dx)=\sum_{i,j\ge1}(L_2^\Delta u_j,e_i)_{\mathcal H}\,a_1^+(e_i)Ca_2^-(u_j),
\end{equation}
where the series converges strongly in $\mathcal L(\mathcal F_{n+1}(\mathcal E),\mathcal F_{m+1}(\mathcal E))$.

Next, we note that $\mathcal H\otimes\mathcal G=L^2(X\to\mathcal G,\sigma)$. Hence, by \eqref{vd5yw573}, $\chi_\Delta(\cdot)L_1(\cdot)\in \mathcal H\otimes\mathcal G$.
Note also that $\mathcal H\otimes\mathcal G$ is a subspace of $\mathcal E^{\otimes 2}=(\mathcal H\oplus\mathcal G)^{\otimes 2}$. For each $m\in\mathbb N$, we denote by $P_m:\mathcal E^{\otimes m}\to\mathcal E^{\odot m}$ the symmetrization operator.  We naturally set, for each $f^{(k)}\in\mathcal F_{k}(\mathcal E)$,
$$\int_\Delta a_1^+(x)a_2^+(L_1(x))\sigma(dx)f^{(k)}=P_{k+2}\big((\chi_\Delta(\cdot)L_1(\cdot))\otimes f^{(k)}\big).$$
Hence, we define
$$\int_\Delta a_1^+(x)Ca_2^+(L_1(x))\sigma(dx)\in\mathcal L(\mathcal F_{n-1}(\mathcal E),\mathcal F_{m+1}(\mathcal E))$$
by
$$\int_\Delta a_1^+(x)Ca_2^+(L_1(x))\sigma(dx)f^{(n-1)}:=P_{m+1}(1_{\mathcal E}\otimes (CP_n))\big(
(\chi_\Delta(\cdot)L_1(\cdot))
\otimes f^{(n-1)}\big) $$
for $f^{(n-1)}\in\mathcal F_{n-1}(\mathcal E)$. Here $1_\mathcal E$ is the identity operator in $\mathcal E$. Therefore, 
\begin{equation}\label{s4qy}
\int_\Delta a_1^+(x)Ca_2^+(L_1(x))\sigma(dx)=\sum_{i,j\ge1}(\chi_\Delta(\cdot)L_1(\cdot)
,e_i\otimes u_j)_{\mathcal H\otimes\mathcal G}\,a_1^+(e_i)Ca_2^+(u_j),
\end{equation}
where the series converges strongly in $\mathcal L(\mathcal F_{n-1}(\mathcal E),\mathcal F_{m+1}(\mathcal E))$.

Similarly to Remark \ref{yre7i645i7r45e}, for $i=1,2$, we define $\mathbb L_i^\Delta\in\mathcal L(\mathcal H,\mathcal G)$ by 
$$\mathbb L_i^\Delta h:=\int_\Delta h(x)L_i(x)\sigma(dx),\quad h\in\mathcal H$$ (in the sense of Bochner integration).

Similarly to \eqref{cw5uu5}, \eqref{vctsw5w3}, we define 
$$\int_\Delta a_2^+ (\mathcal JL_2(x))Ca_1^-(x)\sigma(dx)\in\mathcal L(\mathcal F_{n+1}(\mathcal E),\mathcal F_{m+1}(\mathcal E))$$
by
$$\int_\Delta a_2^+ (\mathcal JL_2(x))Ca_1^-(x)\sigma(dx) f^{\otimes(n+1)}:=(n+1)\big((\mathbb L_2^\Delta)' h\big)
\odot (Cf^{\otimes n}),\quad f=(h,g)\in\mathcal E. $$
Similarly to \eqref{fys6w6u4},
\begin{equation}\label{vctesw5y}
\int_\Delta a_2^+ (\mathcal JL_2(x))Ca_1^-(x)\sigma(dx)=\sum_{i,j\ge1}\big((\mathbb L_2^\Delta)' e_i,u_j\big)_{\mathcal G}\,
a_2^+(u_j)Ca_1^-(e_i),
\end{equation}
where the series converges strongly in $\mathcal L(\mathcal F_{n+1}(\mathcal E),\mathcal F_{m+1}(\mathcal E))$.

 Finally, we define
$$\int_\Delta a_2^- (\mathcal JL_1(x))Ca_1^-(x)\sigma(dx)\in\mathcal L(\mathcal F_{n+1}(\mathcal E),\mathcal F_{m-1}(\mathcal E))$$
by
$$\int_\Delta a_2^- (\mathcal JL_1(x))Ca_1^-(x)\sigma(dx) f^{\otimes(n+1)}
:=(n+1)a^-_2\big((\mathbb L_1^\Delta)'h\big)(Cf^{\otimes n}),\quad f=(h,g)\in\mathcal E.$$
Hence,
\begin{equation}\label{cxrw53ew}
\int_\Delta a_2^- (\mathcal JL_1(x))Ca_1^-(x)\sigma(dx)=\sum_{i,j\ge1}
\big((\mathbb L_1^\Delta)'e_i,u_j\big)_{\mathcal G}
\,a_2^-(u_j)Ca_1^-(e_i),
\end{equation}
where the series converges strongly in $\mathcal L(\mathcal F_{n+1}(\mathcal E),\mathcal F_{m-1}(\mathcal E))$. 
\end{proof}

\begin{proposition}\label{3rtlgpr}
For each $\Delta\in\mathcal B_0(X)$, the particle density $\rho(\Delta)=\int_\Delta A^+(x)A^-(x)\,\sigma(dx)$ is a well-defined  Hermitian operator in $\mathcal F(\mathcal E)$ with domain $\mathcal F_{\mathrm{fin}}(\mathcal E)$. Furthermore, for any $\Delta_1,\Delta_2\in\mathcal B_0(X)$, $[\rho(\Delta_1)$, $\rho(\Delta_2)]=0$. \end{proposition}

\begin{proof} By Proposition \ref{cyse6ui5i}, we have  $\rho(\Delta)\in\mathcal L(\mathcal F_{\mathrm{fin}}(\mathcal E))$. The fact that $\rho(\Delta)$ is a Hermitian operator in $\mathcal F(\mathcal E)$ easily follows from the proof of Proposition~ \ref{cyse6ui5i}.

To prove the commutation, one uses the corresponding Bochner integrals,  formulas 
\eqref{r5w5w32}, \eqref{fys6w6u4}--\eqref{cxrw53ew}. In doing so, one uses the fact that every strongly convergent sequence of bounded linear operators is norm-bounded. Hence, for every strongly convergent sums of bounded linear operators, $A=\sum_{i\ge1}^\infty A_i$ and $B=\sum_{j\ge1}^\infty B_j$, one has $AB=\sum_{i,j\ge 1}A_iB_j$, where the latter double series converges strongly. 
 \end{proof}

\begin{theorem}\label{due6uew4}
The operators $(\rho(\Delta))_{\Delta\in\mathcal B_0(X)}$ defined by Propositions~\ref{cyse6ui5i}, \ref{3rtlgpr}   and the state $\tau$ defined by the vacuum vector $\Omega$ satisfy the assumptions of Theorem~\ref{tes56uwe4u6}. The corresponding point process $\mu$ is the Cox process $\Pi_R$, where $R(x)=|G(x)|^2$, and $G(s)$ is the Gaussian field from Theorem~\ref{reaq5y43wu}. The point process $\Pi_R$ is hafnian with the correlation kernel $\mathbb K(x,y)$ given by \eqref{d6e6ie4}, where $\mathcal K_1(x,y)$ and $\mathcal K_2(x,y)$ are  given by \eqref{w909i8u7y689} and \eqref{ftsqw43qd}, respectively.
\end{theorem}

\begin{proof} Direct calculations show that, for any $\Delta_1,\Delta_2 \in\mathcal B_0(X)$ and $C\in\mathcal L(\mathcal F_{\mathrm{fin}}(\mathcal E))$, formula \eqref{cx4t3y7} holds, which implies formula \eqref{terw5yw35}. We apply the vacuum state  $\tau$ to \eqref{terw5yw35}. Using formulas \eqref{dr6e6u43wq}, \eqref{d6e6ie4}, \eqref{waaq4yq5y}, \eqref{ytd7r57}, \eqref{ytyuryr}, Theorem~\ref{reaq5y43wu}, and Proposition~\ref{cyse6ui5i}, we conclude that  the measure $\theta^{(n)}$ is given by
\begin{align}
\theta^{(n)}(dx_1\dotsm dx_n)&=\frac1{n!}\,\tau\big(\Psi(x_n)\dotsm\Psi(x_1)\Phi(x_1)\dotsm\Phi(x_n)\big)
\sigma^{\otimes n}(dx_1\dotsm dx_n)\notag\\
&=\frac1{n!}\,\mathbb E\big(\overline{G(x_n)}\dotsm \overline{G(x_1)}G(x_1)\dotsm G(x_n)\big)\sigma^{\otimes n}(dx_1\dotsm dx_n)\label{buyftd7s}\\
&=\frac1{n!}\,\mathbb E\big(|G(x_1)|^2\dotsm|G(x_n)|^2\big)\sigma^{\otimes n}(dx_1\dotsm dx_n)\label{vud6ws5}\\
&=\frac1{n!}\,\operatorname{haf}[\mathbb K(x_i,x_j)]_{i,j=1,\dots,n}\,\sigma^{\otimes n}(dx_1\dotsm dx_n).\label{uyfy7de6}
\end{align}
In particular, $\theta^{(n)}$ is a positive measure that is concentrated on $X^{(n)}$.

If $\mathcal Y=G(x)$ or $\overline{G(x)}$ and $\mathcal Z=G(y)$ or $\overline{G(y)}$, then
$$|\mathbb E(\mathcal Y\mathcal Z)|\le\big(\mathbb E(|\mathcal Y|^2)\mathbb E(|\mathcal Z|^2)\big)^{1/2}=\big(\mathbb E(|G(x)|^2)\,\mathbb E(|G(y)|^2)\big)^{1/2}=\|L_1(x)\|_{\mathcal G}\,\|L_1(y)\|_{\mathcal G}. $$ The number of all partitions $\{i_1,j_1\},\dots,\{i_n,j_n\}$ of $\{1,\dots,2n\}$ is $\frac{(2n)!}{(n!)2^n}\le 2^nn!$\,. Hence, by \eqref{buyftd7s} and the formula for the moments of Gaussian random variables,
$$\theta^{(n)}(\Delta^n)\le\bigg(2\int_\Delta\|L_1(x)\|_{\mathcal G}^2\,\sigma(dx)\bigg)^n,\quad\Delta\in\mathcal B_0(X).$$
Thus, the operators $(\rho(\Delta))_{\Delta\in\mathcal B_0(X)}$   satisfy the assumptions of Theorem~\ref{tes56uwe4u6}.

The statement of the theorem about the arising point process $\mu$ follows immediately from \eqref{vud6ws5} and \eqref{uyfy7de6}.
\end{proof}

\subsection*{Acknowledgments}
The authors are grateful to the referee for valuable comments that improved the manuscript.


\begin{thebibliography}{99}

\bibitem{A1} H. Araki, On quasifree states of CAR and Bogoliubov automorphisms, {\it Publ. Res. Inst. Math. Sci.} {\bf 6} (1970/71) 385--442 

\bibitem{A3} H. Araki, On quasifree states of the canonical commutation relations. II, {\it Publ. Res. Inst. Math. Sci.} {\bf 7} (1971/72) 121--152. 

\bibitem{A2} H. Araki,  and M. Shiraishi, ``On quasifree states of the canonical commutation relations,'' {\it I. Publ. Res. Inst. Math. Sci.} {\bf 7}, 105--120 (1971/72). 


\bibitem{A4}  H. Araki and S.  Yamagami, On quasi-equivalence of quasifree states of the canonical commutation relations, {\it Publ. Res. Inst. Math. Sci.} {\bf 18}  (1982) 703--758 (283–338). 

\bibitem{ArWoods} H. Araki and E. Woods, Representations of the C.C.R. for a nonrelativistic infinite free Bose gas, {\it J. Math.\ Phys.}\ {\bf 4} (1963) 637--662.

\bibitem{ArWyss} H. Araki and W. Wyss, Representations of canonical anticommutation relations,'
{\it Helv. Phys. Acta} {\bf 37} (1964), 136--159. 

\bibitem{Barvinok} A. Barvinok,  {\it Combinatorics and Complexity of Partition Functions}  (Springer, Cham, 2016).

\bibitem{Bauer}  H. Bauer, {\it Measure and Integration Theory}   (de Gruyter, Berlin, 2001). 

\bibitem{BM} C. Benard and O. Macchi, Detection and ``emission'' processes of quantum particles in a ``chaotic state,'' {\it J. Mathematical Phys.} {\bf 14} (1973) 155--167.



\bibitem{BK}  Y. M. Berezansky and Y. G. Kondratiev,
{\it Spectral Methods in Infinite Dimensional Analysis} (Kluwer Acad.\ Publ., Dordrecht/Boston/London,
 1994).
 
 \bibitem{BUS} Y. M. Berezansky, Z. G. Sheftel, and G. F. Us, {\it Functional Analysis, Vol.~1}   (Birkh\"auser-Verlag, Basel, 1996).


\bibitem{Berezin} F. A. Berezin, {\it The Method of Second Quantization}. Pure and Applied Physics 24 (Academic Press, New York, 1966). 


\bibitem{Borodin} A. Borodin,  Determinantal point processes, in {\it The Oxford Handbook of Random Matrix Theory}, pp.\ 231--249  (Oxford Univ.\ Press, Oxford, 2011). 

\bibitem{BR} O. Bratteli and D. W. Robinson, {\it Operator Algebras and Quantum Statistical Mechanics. Vol.~2. Equilibrium states. Models in Quantum Statistical Mechanics}. Second edition  (Springer-Verlag, Berlin, 1997).

\bibitem{DG} J. Derezi\'nski and  C. G\'erard,  {\it Mathematics of Quantization and Quantum Fields} 
 (Cambridge University Press, Cambridge, 2013).

\bibitem{E1} N. Eisenbaum and H. Kaspi,  On permanental processes, {\it Stochastic Process. Appl.} {\bf 119} (2009)  1401--1415. 

\bibitem{E2} N. Eisenbaum, Stochastic order for alpha-permanental point processes, {\it Stochastic Process. Appl.} {\bf 122}  (2012)  952--967.

\bibitem{E3} N. Eisenbaum,  Inequalities for permanental processes, {\it Electron. J. Probab.} {\bf 18} (2013), No.\ 99, 15 pp. 

\bibitem{E4} N. Eisenbaum, Permanental vectors with nonsymmetric kernels,'' {\it Ann. Probab.} {\bf 45} (2017) 210--224.

\bibitem{Frenkel} P. E. Frenkel, ``Remarks on the $\alpha$-permanent, {\it Math. Res. Lett.} {\bf 17} (2010)  795--802.


\bibitem{GGPS} G. A. Goldin, J. Grodnik, R. T. Powers, and D. H. Sharp,  Nonrelativistic current algebra in the $N/V$ limit, {\it J. Math.\ Phys.}\ {\bf 15}  (1974)  88--100.

\bibitem{KMR} H. Kogan, M. B. Marcus, and J Rosen,   Permanental processes, {\it Commun. Stoch. Anal.} {\bf 5} (2011)  81--102. 

\bibitem{Koshida} S. Koshida, Pfaffian point processes from free fermion algebras; perfectness and conditional measures, 	{\it SIGMA\/} {\bf 17} (2021), 008, 35 pp.

\bibitem{Lenard}  A. Lenard, Correlation functions and the uniqueness of the state in classical statistical mechanics, {\it Commun. Math. Phys.} {\bf 30} (1973) 35--44. 

\bibitem{Ly} E. Lytvynov,  Fermion and boson random point processes as particle distributions of infinite free Fermi and Bose gases of finite density, {\it Rev. Math. Phys.} {\bf 14} (2002) 1073--1098. 

\bibitem{LM} E. Lytvynov and L. Mei, On the correlation measure of a family of commuting Hermitian operators with applications to particle densities of the quasi-free representations of the CAR and CCR,
{\it J. Funct. Anal.}  {\bf 245} (2007)  62--88. 

\bibitem{Macchi1} O. Macchi,  Distribution statistique des instants d'\'emission des photo\'electrons d'une lumi\'ere thermique, {\it C. R. Acad. Sci. Paris S\'er. A-B\/} {\bf 272} (1971) A437--A440.

\bibitem{Macchi2}   O. Macchi, The coincidence approach to stochastic point processes, {\it Advances in Appl. Probability} {\bf 7} (1975) 83--122. 

\bibitem{MV} J. Manuceau and  A. Verbeure, Quasi-free states of the C.C.R.—algebra and Bogoliubov transformations, {\it Comm. Math. Phys.} {\bf 9} (1968) 293--302.

\bibitem{MR1} M. B. Marcus and J.  Rosen, A sufficient condition for the continuity of permanental processes with applications to local times of Markov processes, {\it Ann. Probab.} {\bf 41} (2013) 671--698. 

\bibitem{MR2}M. B. Marcus and J.  Rosen, Conditions for permanental processes to be unbounded, {\it Ann. Probab.} {\bf 45} (2017)  2059--2086.

\bibitem{MR3} M. B. Marcus and J.  Rosen, Sample path properties of permanental processes. {\it Electron. J. Probab.} {\bf 23} (2018),  Paper No.~58, 47 pp. 

\bibitem{ST}  T. Shirai and Y. Takahashi,  Random point fields associated with certain Fredholm determinants. I. Fermion, Poisson and boson point processes, {\it J. Funct. Anal.} {\bf 205} (2003) 414--463. 

\bibitem{Surgailis} D. Surgailis, On multiple Poisson stochastic integrals and associated Markov semigroups,  {\it Probab.\ Math.\ Statist.} {\bf 3} (1984)  217--239.



\bibitem{VJ} D. Vere-Jones,   A generalization of permanents and determinants,  {\it Linear Algebra Appl.} {\bf 111} (1988) 119--124.


\bibitem{VJ2} D. Vere-Jones, Alpha-permanents and their applications to multivariate gamma, negative binomial and ordinary binomial distributions, {\it New Zealand J. Math.} {\bf 26} (1997),  125--149. 


\end{thebibliography}
\end{document}